\documentclass{amsart}
\usepackage[colorlinks=true, allcolors=blue]{hyperref}
\usepackage{amsmath,amsfonts,amssymb}
\usepackage{amsthm}

\title[Invariants of Finite Orthogonal Groups]{Invariants of Finite Orthogonal Groups of Plus Type in Odd Characteristic}
\author{H.E.A. Campbell, R. James Shank and David L. Wehlau}

\address{Department of Mathematics and Statistics, Queen's 
University, Kingston ON, Canada}
\email{eddy@unb.ca}
\address{School of Engineering, Mathematics and Physics, 
University of Kent,
Canterbury, United Kingdom}
\email{R.J.Shank@kent.ac.uk}
\address{Department of Mathematics and Computer Science, Royal Military College of Canada, Kingston ON, Canada}
\email{wehlau@rmc.ca}

\subjclass{Primary 13A50; Secondary 20F55}

\keywords{modular invariant theory, finite classical groups, orthogonal group}

\newcommand{\field}{\mathbb{F}}
\newcommand{\gl}[2]{{{\text{GL}_{{#1}}(\field_{#2})}}}
\newcommand{\orthp}[2]{{{\text{O}^{+}_{{#1}}(\field_{#2})}}}
\newcommand{\sorthp}[2]{{{\text{SO}^{+}_{{#1}}(\field_{#2})}}}

\newcommand{\invring}{S_m^{G_m}}
\newcommand{\PP}{\mathcal{P}}
\newcommand{\V}{\mathcal{V}}
\newcommand{\cS}{\mathcal{S}}
\newcommand{\F}{\mathcal{F}}
\newcommand{\HH}{\mathcal{H}}
\newcommand{\B}{\mathcal{B}}
\newcommand{\BH}{\mathcal{B}_H}
\newcommand{\A}{\mathcal{A}}
\newcommand{\R}{\mathcal{R}}
\newcommand{\T}{\mathcal{T}}
\newcommand{\Z}{\mathbb{Z}}
\newcommand{\tat}{t\^ete-a-t\^ete}
\newcommand{\binomial}[2]{\genfrac{(}{)}{0pt}{}{#1}{#2}}

\newtheorem{thm}[subsection]{Theorem}
\newtheorem{lem}[subsection]{Lemma}
\newtheorem{cor}[subsection]{Corollary}
\newtheorem{remark}[subsection]{Remark}
\newtheorem{example}[subsection]{Example}

\DeclareMathOperator{\Span}{Span}
\DeclareMathOperator{\lt}{LT}
\DeclareMathOperator{\LM}{LM}
\DeclareMathOperator{\image}{image}
\DeclareMathOperator{\sgn}{sgn}
\DeclareMathOperator{\Fr}{Fr}
\DeclareMathOperator{\tr}{Tr}
\DeclareMathOperator{\Cat}{{Cat}}
\newcommand{\wdeg}{{S\mbox{-}\deg}}
\newcommand{\rdeg}{{R\mbox{-}\deg}}

\begin{document}
\begin{abstract}
    We describe the rings of invariants for the finite orthogonal groups of plus type
in odd characteristic acting on the defining representations.  
We also describe the invariants of the corresponding Sylow subgroups in the defining characteristic.  
In both cases we construct minimal algebra generating sets and describe the relations among the generators.  
Both rings of invariants are shown to be complete intersections and thus are 
Cohen-Macaulay.  We expect the techniques we use will generalise to give a systematic 
computation for rings of invariants for all of the finite classical groups in odd characteristic.
\end{abstract}

\maketitle
\setcounter{tocdepth}{1}   
\tableofcontents


\section{Introduction}
The fundamental problem in the invariant theory of finite groups is to determine the ring of invariants 
of a representation of a finite group.  The celebrated theorem of Shepherd--Todd and Chevalley proves that
such a ring of invariants in characteristic zero is a polynomial ring if and only if the group action is 
generated by pseudo-reflections (group elements which fix pointwise a hyperplane). 
An excellent survey of the characteristic zero results and theory in characteristic zero
is the article by Stanley~\cite{stanley:79}.

In positive characteristic the situation is much more complex.
In 1911, L.E.~Dickson~\cite{dickson-fundsystinvagene:11} gave an explicit description of the ring of 
invariants of the general linear group over any finite field.  In the past few decades
there has been significant progress in determining the rings of invariants of the defining 
representations of the finite classical groups.  
The rings of invariants for the symplectic groups were computed by David Carlisle and Peter Kropholler in the 1990s, see 
\cite[\S 8.3]{benson-polyinvafinigrou:93}.
The invariants in the unitary case were computed by 
Huah Chu and Shin-Yao Jow \cite{Chu+Jow-polyinvunitary:06}.
For the orthogonal groups, while there are computations in a number of special cases (see 
\cite{chiang+hung-invorth:93},
\cite{chu-polyinvorth:01},
\cite{cohen-ratfuncortho:90},
\cite{Kropholler+Rajaei+Segal:05} and
\cite{smith-ringinvorth:99}
)
there is no general result.  
Almost all of the defining representations for these groups are generated by pseudo-reflections, but the rings of invariants are rarely polynomial rings. 

We work over the finite field $\field_q$ of order $q$ and characteristic $p$ where $p$ is an odd prime.
We compute the ring of invariants for the
finite orthogonal group $\orthp{2m}{q}$ of plus type
acting on its defining representation $V$ of dimension $n=2m$.  

As an intermediate result, we compute the ring of invariants for 
a Sylow $p$-subgroup. 
For both of these rings we construct a minimal set of algebra generators and describe
generators for the ideal of relations. 
Both of these rings are complete intersections, see Theorems~\ref{oplus_thm_v2} and \ref{sylow_thm}, and are thus Cohen-Macaulay.  We note that for modular 
representations, rings of invariants are rarely Cohen-Macaulay, let alone 
complete intersections.

  For a Sylow $p$-subgroup $P$ of $\orthp{2m}{q}$ we find a set $\HH_0$ consisting
of $n$ orbit products of linear forms such that 
$\HH_0$ forms a homogeneous system of parameters.  Using a family of 
$\orthp{2m}{q}$ invariants $\xi_1,\xi_2,\dots,\xi_{n-1}$, we show that
$$\field_{q}[V]^{P} = \bigoplus_{\gamma\mid\Gamma_0} \field_{q}[\HH_0]\gamma$$ 
where the sum is over all monomials $\gamma$ dividing 
$$\Gamma_0=\prod_{i=1}^{n-2} \xi_{i}^{e_i} \text{ where } e_i = q^{m-\lceil i/2\rceil}-1$$
(see Corollary~\ref{Sylow_basis}).
In particular, $\field_{q}[V]^{P}$is Cohen-Macaulay and this implies that
$\field_{q}[V]^{\orthp{2m}{q}}$ is Cohen-Macaulay.
We also give a Khovanskii basis for the Sylow invariants
(see the definition at the end of the next section).

We show that there are homogeneous orthogonal invariants 
$d_{1,m}, d_{2,m},\ldots,d_{m,m}$ such that
$\HH=\{\xi_1,\ldots,\xi_{m},d_{1,m},\ldots,d_{m,m}\}$
is a homogeneous system of parameters.   
Furthermore, we show that 
$$\field_{q}[V]^{\orthp{2m}{q}} = \bigoplus_{\gamma\in\B} \field_q[\HH]\gamma$$ 
where $\B$ is the set of monomials dividing 
$$\Gamma= \prod_{i=m+1}^{n-1}\xi_{i}^{q^{n-i}-1}$$
(see Theorem~\ref{oplus_thm_v2}).
For the orthogonal invariants, the ring is generated by two elements as 
an algebra over the Steenrod algebra (see Section \ref{st_sec} and 
Theorem \ref{st_alg_gen}).

Since the $P$-action is triangular,  there exists a finite Khovanskii basis for $\field_q[V]^{P}$ (see \cite{shank+wehlau-compmoduinva:02a}).  Hence we approximate $\field_q[V]^{P}$ using its lead term algebra.  However, in general, rings of invariants do not have finite Khovanskii bases.
For example, for a permutation representation, the ring of invariants has a finite 
Khovanskii basis if and only if the group is generated by reflections, see \cite{thiery+thomasse:04}.
We do not expect $\field_{q}[V]^{\orthp{2m}{q}}$ to have a finite Khovanskii basis.
Instead we approximate $\field_{q}[V]^{\orthp{2m}{q}}$ 
    by the associated graded algebra arising from a filtration given by a certain valuation.

For both the orthogonal group and its Sylow $p$-subgroup, we express the ring of invariants
as a free module over a homogeneous system of parameters with basis given by all monomial
factors of a single monomial.  Such a basis is called a block basis (see \cite{campbell+hughesetal-baseringcoin:96}).  
Since the rings are Cohen-Macaulay, having block
bases implies that these rings are both complete intersections. 

We believe that the methods we use here generalise to give a systematic computation for rings of invariants for all of the finite classical groups, at least for odd characteristic; this is work in progress. 
For background on the finite classical groups we suggest
\cite{taylor-classicalgroups:92}.
For background on invariant theory see 
\cite{benson-polyinvafinigrou:93},
\cite{Campbell+Wehlau:MIT11},
\cite{derksen+kemper-compinvatheo:02} or
\cite{Neusel+Smith-Invatheofinigrou:02}.

\section{The Setting}\label{setting_section}
For a symmetric bilinear form $\beta$ on a vector space $V$,
the orthogonal group ${\rm O}(\beta)$ is the subgroup of the the general linear group ${\rm GL}(V)$ consisting of linear transformations $g$
satisfying $\beta(gv,gu)=\beta(v,u)$ for all $u,v\in V$.
In this paper we work over the finite field $\field_q$
of order $q$ and characteristic $p$ where $p$ is an odd prime.

In this context, if the dimension of $V$ is even, there are two equivalence classes of symmetric non-degenerate
bilinear forms, one of plus type and one of minus type
(see Chapter 11 of \cite{taylor-classicalgroups:92}).
In this paper we focus on the plus type. 
We expect the generalisation of our methods to the minus type to be relatively straightforward. 

We denote the dimension of $V$ by $n=2m$.
The quadratic form $\xi_1$ associated to the bilinear form $\beta$ is defined by $\xi_1(v):=\beta(v,v)/2$ for $v\in V$.
We choose an ordered basis for $V$, 
$[e_m,e_{m-1}\ldots,e_1,f_1,f_2,\ldots, f_m]$, 
so that $\beta(e_i,f_i)=1$, $\beta(e_i,e_j)=\beta(f_i,f_j)=0$
and, if $i\not=j$, $\beta(e_i,f_j)=0$, in other words,
the basis consists of $m$ orthogonal hyperbolic pairs
(see pages 56 and 139 of \cite{taylor-classicalgroups:92}).
We note that $m$ is the Witt index and the basis vectors are isotropic.
We let $[y_m,y_{m-1},\ldots,y_1,x_1,x_2,\ldots,x_m]$ denote the dual basis and observe that the quadratic form is given by 
$$\xi_1=y_mx_m+\cdots +y_1x_1\ .$$

The left action of ${\rm GL}(V)$ on $V$ induces a right action on the dual $V^*$ given by $(\phi\cdot g)(v)=\phi(g\cdot v)$ for $\phi\in V^*$, 
$g\in {\rm GL}(V)$ and $v\in V$. This action extends to an action by algebra automorphisms on $\field_q[V]$, the symmetric algebra on $V^*$.
Choosing a basis for $V$ allows us to identify ${\rm GL}(V)$
with the matrix group $\gl{n}{q}$ and
identify $\field_q[V]$ with
the polynomial algebra on $n=2m$ variables
$$S_m:=\field_q[y_m,\ldots,y_1,x_1,\ldots,x_m]\ .$$
The right action of 
$\gl{n}{q}$ on $S_m$ is given by identifying $y_m$ with 
$[1\, 0 \,\cdots\, 0]$, $y_{m-1}$ with $[0\, 1\, 0 \,\cdots\, 0]$, etc.
The orthogonal group of plus type is 
$$\orthp{2m}{q}=\{g\in \gl{n}{q}\mid \xi_1\cdot g=\xi_1\}.$$
Its order is  given by 
$$|\orthp{2m}{q}| = 2 q^{m(m-1)}(q^m-1)\prod_{i=1}^{m-1}(q^{2i}-1),$$
see \cite[page 141]{taylor-classicalgroups:92}.
For notational simplicity, we write $G_m:=\orthp{2m}{q}$.
The ring of invariants is 
$$\invring=\{f\in S_m\mid f\cdot g=f, \,\forall g\in G_m\}.$$
Define an algebra map $\psi_{[m,j]}:S_m\to S_m$ by
$$\psi_{[m,j]}(a):=\prod_{u\in U_j}(a-u)$$
for $\deg(a)=1$, where $U_j = \Span_{\field_q}\{x_{m-j+1},\ldots,x_m \}$.
Note that $\psi_{[m,1]}(a)=a^q-ax_m^{q-1}$ and 
the height of $\ker(\psi_{[m,j]})$ is $j$.

Clearly $\xi_1\in \invring$.
For $i\geq 2$, define 
$$\xi_{i}:=y_m^{q^{i-1}}x_m+y_mx_m^{q^{i-1}}+\cdots +y_1^{q^{i-1}}x_1 +y_1x_1^{q^{i-1}}\ .$$
Since taking the $q^{th}$ power is $\field_q$-linear, $\xi_{i}\in \invring$. 
To avoid cluttering the notation we suppress the $m$ from the notation for the $\xi_{i}$.

For an integral domain $A$, let $\F(A)$ denote the field of fractions. The following is due to Carlisle and Kropholler \cite{carlisle+kropholler-ratiinvacertorth:92}.

\begin{thm} $\F(\invring)=\field_q(\xi_1,\ldots,\xi_{n})$.
\end{thm}

Throughout, we order the variables taking
$y_m>y_{m-1}>\cdots >y_1>x_1>\ldots >x_m$.
There are many ways of extending this ordering to a monomial order on $S_m$. We will use both the lexicographic and the graded reverse lexicographic (referred to as grevlex) orders. 
For $f\in S_m$, we denote the lead term of $f$, with respect to whichever order is being considered, by $\lt(f)$.
For a subalgebra $A\subset S_m$, $\lt(A)$ denotes the algebra generated by the lead terms of the elements of $A$.
A Khovanskii basis for $A$ is a subset of $A$, say $\B$, such that
$\{\lt(f)\mid f\in \B\}$ is a generating set for $\lt(A)$.
A Khovanskii basis is a nice generating set for $A$.
Khovanksii bases were previously known as SAGBI bases.
For more on Khovanskii(SAGBI) bases see 
\cite[\S 5.1]{{Campbell+Wehlau:MIT11}}.

\section{Missing Invariants}

Define $\overline{V}:=V\otimes\overline{\field}_q$
where $\overline{\field}_q$ is the algebraic closure of
$\field_q$. For $X\subset S_m$, the variety associated to $X$ is given by
$$\V(X):=\{v\in \overline{V}\mid f(v)=0, \, \forall f\in X\} \subseteq \overline{V}.$$
Let $(\invring)_+$ denote the augmentation ideal, i.e., the set of  homogeneous invariants of positive degree. Since $G_m$ is a finite group, 
$\V((\invring)_+)=\{\underline{0}\}$.
Recall that $[e_m,\ldots,e_1,f_1,\ldots,f_m]$ is the ordered basis for $V$ 
dual to the ordered basis for $V^*$ given by $[y_m,\ldots,y_1,x_1,\ldots,x_m]$. 
In the following, we write
$\V(\xi_1,\ldots,\xi_{i})$ for 
 $\V(\{\xi_1,\ldots,\xi_{i}\})$.

\begin{thm}\label{orth_var} For $s\geq m$,
$$\V(\xi_1,\ldots,\xi_s)=\bigcup_{g\in G_m}
g\cdot\Span_{\overline{\field}_q}\{e_1,\ldots,e_m\}.$$
\end{thm}
\begin{proof} 
Observe that the variety $\V(\xi_1,\ldots,\xi_s)$ is closed under scalar multiplication 
since it is defined by homogeneous polynomials.
Let $W_m$ denote 
$\Span_{\overline{\field}_q}\{e_1,\ldots,e_m\}$.
Since $\xi_{i}(v)=0$ for $v\in W_m$ and $\xi_{i}\in \invring$, 
$$\bigcup_{g\in G_m}g\cdot W_m\subseteq \V(\xi_1,\ldots,\xi_s).$$ 
Therefore, it is sufficient to prove that for 
$v\in \V( \xi_1,\ldots,\xi_{m})$, there exists $g\in G_m$ such that
$gv\in W_m$.

The proof is by induction on $m$. For $m=1$, we have $x_1(v)y_1(v)=0$. If $x_1(v)\not=0$, then take $g$ to be the transposition which interchanges $x_1$ and $y_1$.

For $m>1$, define $\bar{v}:=v-y_m(v)e_m-x_m(v)f_m$, let
$\bar{\xi_{i}}$ denote the restriction of $\xi_{i}$ to 
$\Span_{\overline{\field}_q}\{e_1,\ldots,e_{m-1},f_{m-1},\ldots f_1\}$  and
identify $G_{m-1}$ with the point-wise stabiliser of  
${\rm Span}_{\field_q}\{e_m,f_m\}$ in $G_m$.

If $x_m(v)=0$ then $\bar{\xi_{i}}(\bar{v})=0$ for $i=1,\ldots,m$.
By induction, there is an element $g\in G_{m-1}< G_m$
with $g\bar{v}\in W_{m-1}$. Observe that 
$gv\in W_m$.

For $u\in \overline{V}$ we use $\Fr^i(u)$ to denote the $i^{th}$ iteration of the Frobenius map on $u$.
Recall that $\beta(u,u)=2\xi_1(u)$ and note that, for $i\geq2$,
$\beta(u,\Fr^{i-1}(u))=\xi_{i}(u)$.

If $x_m(v)\not=0$ then we scale $v$ so that $x_m(v)=1$ and define $w:=v-\Fr(v)$.
Observe that
\begin{eqnarray*}
\xi_1(w)&=&
\beta(v-\Fr(v),v-\Fr(v))/2
=\left(\beta(v,v)+\beta(\Fr(v),\Fr(v))-2\beta(v,\Fr(v))\right)/2\\
   &=&\xi_1(v)+\xi_1(v)^q-\xi_2(v)=0
\end {eqnarray*}
and, for $i\geq2$,
\begin{eqnarray*}
\xi_{i}(w)&=&\beta(v-\Fr(v),\Fr^{i-1}(v-\Fr(v)))=\beta(v-\Fr(v),\Fr^{i-1}(v)-\Fr^{i}(v))\\
&=&\beta(v,\Fr^{i-1}(v))-\beta(\Fr(v),\Fr^{i-1}(v))-\beta(v,\Fr^{i}(v))+\beta(\Fr(v),\Fr^{i}(v))\\
   &=&\xi_{i}(v)-\xi_{i-1}(v)^q-\xi_{i+1}(v)+\xi_{i}(v)^q=0.
\end {eqnarray*}
Therefore $\xi_{i}(w)=0$ for $i=1,\ldots,m-1$.
Since $x_m(w)=0$,
we have $\bar{\xi_{i}}(\bar{w})=0$ for $i=1,\ldots, m-1$ and so by induction
there is an element $g\in G_{m-1}<G_m$
with $g\bar{w}\in W_{m-1}$ and $gw\in W_m$.
Hence $(x_j-x_j^q)(gv)=x_j(gw)=0$ for $j=1,\ldots,m$ and $x_j(gv)\in\field_q$.
Since $g$ stabilises ${\rm Span}_{\field_q}\{e_m,f_m\}$, we have $x_m(gv)=1$.
For convenience, define $c_j=x_j(gv)\in\field_q$ for $j=1,\ldots,m-1$ and
let $h$ denote the linear transformation
given by $y_jh=y_j$ for $j=1,\ldots,m-1$, 
and $y_mh=y_m+\sum_{j=1}^{m-1} c_jy_j$, and $x_jh=x_j-c_jx_m$ for $j=1,\ldots,m-1$
and $x_mh=x_m$. Observe that $h\in G_m$
and, since  $\xi_1(gv)=0$, we have $y_m(hgv)=y_m(gv)+\sum_{j=1}^{m-1} c_jy_j(gv)=0$.
Let $\sigma $ denote the transposition which exchanges $x_m$ and $y_m$. We then have
$x_m(\sigma h g v)=0$ and we can apply the induction argument as above.
\end{proof}

It follows from this theorem that we are missing at least $m$ invariants.

\begin{remark} A vector $v\in\overline{V}$ is isotropic if $\beta(v,v)=0$.
Since $q$ is odd, $\xi_1(v)=\beta(v,v)/2$. Therefore the set of isotropic vectors
in $\overline{V}$ is $\V(\xi_1)$. For $m>1$, using Theorem~\ref{orth_var}, 
$\V(\xi_1,\ldots,\xi_m)$ is a proper subset of $\V(\xi_1)$.
The group $G_m$ acts transitively on the isotropic vectors of $V$;
for $m>1$ this follows from \cite[Lemma 11.27]{taylor-classicalgroups:92}.
Therefore $G_me_1=\V(\xi_1)\cap V$. It then follows from Theorem~\ref{orth_var}
that the set of isotropic vectors in $V$ is
$$\V(\xi_1,\ldots,\xi_m)\cap V=\bigcup_{g\in G_m}
g\cdot\Span_{{\field}_q}\{e_1,\ldots,e_m\}.$$
\end{remark}

\subsection{Steenrod Operations} \label{st_sec}

The Steenrod operations are cohomology operations widely used in algebraic topology.
They can be sensibly defined on $S_m$ and restrict
to operations on $\invring$,  see, for example, \cite{smith-algintrosteenrod:07}. 
Thus applying a Steenrod operation to an invariant produces a potentially new invariant.
The {\it complete Steenrod operator} $\PP(t):S_m\to S_m[t]$ is the algebra homomorphism determined by $\PP(t)(v)=v+v^q t$ for $v$ homogeneous of degree one. Since the map is linear in degree one, $\PP(t)$ is well-defined. For $f$ homogeneous of degree $d$, the Steenrod operations 
$\PP^i(f)$ are defined by
$$\PP(t)(f)=\sum_{i=0}^d\PP^i(f)t^i.$$
Note that for $i>d$ or $i<0$, $\PP^i(f)=0$.
It is clear that $\PP^0(f)=f$ and $\PP^d(f)=f^q$,
i.e., the {\it stability} property is satisfied.
The Steenrod operations satisfy the {\it Cartan identity}:
for $f_1,f_2\in S_m$
$$\PP^i(f_1f_2)=\sum_{j=0}^i \PP^{j}(f_1)\PP^{i-j}(f_2).$$
The Steenrod operations also satisfy the {\it Adem relations}: 
for $i<qj$
$$\PP^i\PP^j=\sum_k (-1)^{i+k}
\binom{(q-1)(j-k) -1}{i-qk}\PP^{i+j-k}\PP^{k}.$$

We can extend the action of $\gl{n}{q}$ to $S_m[t]$
by taking $tg=t$ for all $g$. Using this action,
since taking a $q^{th}$ power is linear in $S_m$, 
we see that $\PP(t)$ commutes with the $\gl{n}{q}$-action
and $\PP^ig=g\PP^i$ for all $i$.

The following lemma is a consequence of the Cartan identity.

\begin{lem}\label{steenrod on q powers} For $f\in S_m$, we have $\PP^i(f^q)=0$ unless $q$ divides $i$,
in which case $\PP^i(f^q)=(\PP^{i/q}(f))^q$.
\end{lem}

\begin{lem}\label{steenrod on linear forms}
    Suppose $v,f\in S_m$ with $v$ homogeneous of degree one.
    Then $v$ divides $\PP^i(vf)$ for all $i$. Moreover, if $f$ is a product of pairwise relatively prime linear forms, then $f$ divides $\PP^i(f)$.
\end{lem}
\begin{proof}
    By definition, $v$ divides $\PP^j(v)$. Therefore, using the Cartan identity, $v$ divides $\PP^i(vf)$.  The second assertion follows.
\end{proof}

 \begin{cor} \label{com_st}(a) $\PP(t)(\xi_1)=\xi_1+\xi_2 t+\xi_1^q t^2$.\\
 (b) $\PP(t)(\xi_2)=\xi_2+2\xi_1^q t+\xi_3 t^q+\xi_2^q t^{q+1}$.\\
(c) For $i\geq 3$, 
 $\PP(t)(\xi_{i})=\xi_{i}+\xi_{i-1}^q t+\xi_{i+1} t^{q^{i-1}}+\xi_{i}^q t^{q^{i-1}+1}$.
 \end{cor}

\begin{lem} \label{sub_max} Suppose $b\in S_m$ is homogeneous of degree $j>0$ and $a\in S_m$ is homogeneous of degree $i>0$.\\
 (a) $\PP^{i+j-1}(ab)=a^q\PP^{j-1}(b)+b^q\PP^{i-1}(a)$.\\
 (b) $\PP^{j+1}(\xi_1b)=\xi_2b^q+\xi_1^q\PP^{j-1}(b)$.\\
 (c) $\PP^{j+3}(\xi_1^2b)=2\xi_1^q\xi_2b^q+\xi_1^{2q}\PP^{j-1}(b)$.\\
 (d) $\PP^{2q+j-1}(\xi_1^qb)=\xi_1^{q^2}\PP^{j-1}b$.
 \end{lem}
 
 \begin{proof} Part (a) follows from the Cartan identity and the stability condition. Parts (b), (c) and (d) follow from (a) and Corollary \ref{com_st}.
 \end{proof}

\subsection{New Invariants}
Let $P_m$ denote the intersection of $G_m$ with the upper triangular unipotent subgroup of $\gl{n}{q}$. 
$P_m$ is a choice of Sylow $p$-subgroup for $G_m$.
For $v\in S_m$ homogenous of degree one, define $N(v)$ to be the orbit product of $v$ over $P_m$. Since $x_m\in S_m^{P_m}$, $N(x_m)=x_m$.
We will see in Section \ref{sylow_section} that 
$\deg(N(x_i))=q^{m-i}$ and $\deg(N(y_i))=q^{m+i-2}$.
Define $$u_m:=\prod_{i=1}^m N(y_i)N(x_i)$$
and observe that $\deg(u_m)=(q^{m-1}+1)(1+q+\cdots +q^{m-1})$.  
\begin{remark}
Observe that the linear factors of $u_m$ lie in a single
$G_m$-orbit. Furthermore, that orbit consists of the non-zero scalar
multiples of the linear factors of $u_m$ 
(see the proof of Theorem~\ref{min_poly_thm}).
\end{remark}

Since $u_m$ is a product of pairwise relatively  prime linear forms, Lemma~\ref{steenrod on linear forms}
implies that $u_m$ divides $\PP^j(u_m)$ in $S_m$.
Define $e(i,m):=\sum_{j=1}^i q^{n-1-j}$ with $n=2m$  and, for $1\leq i \leq m$, define $d_{i,m}\in S_m$ by 
$$u_m d_{i,m}=\PP^{e(i,m)}(u_m ).$$
For convenience, define $d_{0,m}=1$. Observe that $e(i,m)=q^{n-2}+q e(i-1,m-1)$ and that the degree of $d_{i,m}$ is $q^{n-1}-q^{n-1-i}$.

Let $\psi:S_m\to S_m[t]$ denote the algebra homomorphism determined by $\psi(v)=v^q-vt^{q-1}$ for $v$ homogeneous of degree one.
Since the map is linear in degree one, $\psi$ is well-defined.
By comparing $\psi$ and the complete Steenrod operator $\PP(t)$, we see that, for $f$ homogeneous of degree $d$,
$$\psi(f)=\sum_{\ell=0}^d\PP^{d-\ell}(f)(-t^{q-1})^{\ell}.$$
Since $u_m$ divides $\PP^i(u_m)$ in $S_m$, we see that $\psi(u_m)/u_m\in S_m[t]$ is monic with degree $(q-1)\deg(u_m)$ as a polynomial in $t$.

\begin{thm} \label{min_poly_thm} $\psi(u_m)/u_m$ is the minimal polynomial of $x_m$ over the field $\F(S_m^{G_m})$.
In particular, $d_{i,m}\in S_m^{G_m}$ for $i=1,\ldots,m$.
\end{thm}
\begin{proof} If we identify $G_{m-1}$ with the pointwise stabiliser of $\{y_m,x_m\}$ in $G_m$, then the pointwise stabiliser of $x_m$ coincides
with $G_{m-1}H$, where $H$ is the associated Hook group (see Section \ref{HgroupSec}). From this we see that the size the orbit of $x_m$ is
$$\frac{|\orthp{n}{q}|}{|\orthp{n-2}{q}| |H|}=\frac{(q^m-1)(q^{2m-2}-1)}{q^{m-1}-1}=(q^m-1)(q^{m-1}+1)=(q-1)\deg(u_m)$$
and this orbit consists of the non-zero scalar multiples of the linear factors of $u_m$. If $v$ is a linear factor of $u_m$, then the roots of $\psi(v)$ are the non-zero scalar multiples of $v$. From this, using the fact that $\psi(u_m)/u_m$ is monic of degree $(q-1)\deg(u_m)$,
 we conclude that 
$\psi(u_m)/u_m=\prod\{t-x_m g\mid g\in G_m\}\in S_m^{G_m}[t]$. Therefore, $\psi(u_m)/u_m$ is the minimal polynomial of $x_m$ over $\F(\invring)$ and, since the $d_{i,m}$
are coefficients of $\psi(u_m)/u_m$ up to sign, we have $d_{i,m}\in S_m^{G_m}$.
\end{proof}

\begin{lem}\label{lex_lt_lem} Using the lexicographic order on $S_m$, 
$$\lt(u_m)=\prod_{j=1}^m y_j^{q^{m+j-2}}x_j^{q^{m-j}}$$ 
and $\lt(d_{1,m})=y_m^{q^{n-1}-q^{n-2}}$.
\end{lem}
\begin{proof} The expression for $\lt(u_m)$ follows directly from the definition of $u_m$.
We apply $\PP^{q^{n-2}}$ to $u_m$ to get $u_m d_{1,m}$. Using the Cartan identity this gives
$$u_m d_{1,m}=\PP^{q^{n-2}}\left(N(y_m)\right)x_m\prod_{i=1}^{m-1} N(y_i)N(x_i)+F$$
where $\deg_{y_m}(F)<q^{n-1}$. Therefore, since $\lt(\PP^{q^{n-2}}(N(y_m))=y_m^{q^{n-1}}$,
we see that $\lt(u_md_{1,m})=y^{q^{n-1}-q^{n-2}}\lt(u_m)$.
Dividing by $u_m$ gives $\lt(d_{1,m})=y_m^{q^{n-1}-q^{n-2}}$.
\end{proof}

Define $\HH:=\{\xi_1,\ldots,\xi_{m},d_{1,m},\ldots,d_{m,m}\}$.

\begin{thm}  \label{ortho_hsop_thm} $\HH$ is a homogeneous system of parameters.
\end{thm}

\begin{proof} We will show that the variety  in $\overline{V}=\overline{\field}_q\otimes V$
cut out by the ideal generated by $\HH$ is $\{\underline{0}\}$. 
Using Theorem \ref{orth_var},
$$\V( \xi_1,\ldots,\xi_{m})=\bigcup_{g\in \orthp{2m}{q}}g\V(x_1,x_2,\ldots,x_m).$$
For $v\in \V(\xi_1,\ldots,\xi_{m})$, choose $g\in \orthp{2m}{q}$ so that $gv\in \V(x_1,x_2,\ldots,x_m)$. 
From Theorem \ref{min_poly_thm}, $d_{i,m}\in S_m^{G_m}$.
Therefore $d_{i,m}(gv)=d_{i,m}(v)$. Hence to show $\V(\HH)=\{\underline{0}\}$, it is sufficient to show
$\V(x_1,\ldots,x_m,d_{1,m}, \ldots,d_{m,m})=\{\underline{0}\}$.
To do this we consider $d_{i,m}$  modulo the ideal $I:=\langle x_1,x_2,\ldots,x_m\rangle$.

Define $\bar{d}_i\in\field_q[y_1,\ldots,y_m]$ by $\bar{d}_i\equiv_I d_{i,m}$. We will show that $\{\bar{d}_1, \ldots, \bar{d}_m\}$ is an homogeneous system of parameters in
$\field_q[y_m,y_{m-1},\ldots,y_1]$. 
Define $W_m:={\rm Span}_{\field_q}\{y_m,\ldots,y_1\}$.
Let $d_i(W_m)$ denote the $i^\text{th}$ Dickson invariant
in the variables $y_m,\ldots,y_1$.
For a definition and background material on the Dickson invariants see
\cite{Wilkerson-primDickinva:83} or \cite{dickson-fundsystinvagene:11} or
\cite[\S 3.3]{{Campbell+Wehlau:MIT11}}.
We claim that $$d_{i,m}\equiv_I\pm d_i(W_m)^{q^{m-1}}$$ for all $i$.
From this, it follows that $\{\bar{d}_1,\ldots,\bar{d}_m\}$ is an homogeneous system of parameters for $\field_q[y_m,\ldots,y_1]$, as required.

To see that $d_{i,m}\equiv_I\pm d_i(W_m)^{q^{m-1}}$, we look at $\psi(u_m)/u_m$ modulo $I$.
Using Theorem~\ref{min_poly_thm}, its proof and the action of the Hook group given in Section~\ref{sylow_section}, we have
$$\psi(u_m)/u_m =\prod\{t-x_m g\mid g\in G_m\} \equiv_I t^{q^m-1}\left(\prod_{w\in W_m\setminus \{0\}} (t-w)\right)^{q^{m-1}}.$$
On the right-hand side of this equivalence, $d_i(W_m)^{q^{m-1}}$ is the coefficient of
$$t^{q^m-1}(t^{q^{m-i}-1})^{q^{m-1}}=t^{q^{2m-i-1}+q^m-q^{m-1}-1}.$$
On the left-hand side of the equivalence, $d_{i,m}$ is the coefficient of $(-t^{q-1})^\ell$ with
\begin{eqnarray*}
\ell(q-1)&=&(q^{m-1}+1)(q^m-1)-(q-1)e(i,m)\\
&=&q^{2m-1}+q^m-q^{m-1}-1-(q^{n-1}-q^{n-i-1})\\
&=&q^{n-i-1}+q^m-q^{m-1}-1
\end{eqnarray*}
from which we conclude $d_{i,m}\equiv_I\pm d_i(W_m)^{q^{m-1}}$, where the parity is determined by $(-1)^{\ell}=(-1)^{i+1}$.
\end{proof}

\begin{remark}
    We observe that the orthogonal invariants $d_{i,m}$ 
are in certain ways analogous to the Dickson invariants $d_i(V^*)$
which generate
$\field_q[V]^{\text{GL}(V)}$.  In particular, in both cases
these invariants are certain elementary symmetric functions 
in the orbit of a linear form.  Furthermore (see Theorem~\ref{st_alg_gen})
these two families of invariants may each be viewed as arising from
a single invariant under the action of the Steenrod algebra.  In the analogous calculation for the invariants of the symplectic 
group the $d_i(V^*)$ play the role of the $d_{i,m}$.   
\end{remark}

\begin{remark} It follows from the proof of Theorem \ref{ortho_hsop_thm} that using the grevlex order on $S_m$,
$$\LM(d_{i,m})=\LM(d_i(W_m))^{q^{m-1}}=\prod_{j=m-i+1}^m y_j^{q^{m+j-1}-q^{m+j-2}}.$$
(The lead monomial of $d_i(W_m)$ can be computed using 
\cite[Proposition~1.3(b)]{{Wilkerson-primDickinva:83}}). 
\end{remark}

\section{Formulating the Main Theorem}
Define $R_i:=\field_q[T_i,T_{i-1},\ldots, T_1]$ and 
let $\Phi_{i,m}:R_i\to S_m$ denote the algebra map defined 
by $\Phi_{i,m}(T_j)=\xi_{j}$.
If we grade $R_i$ by assigning $T_j$ degree $q^{j-1}+1$, then $\Phi_{i,m}$ preserves degree.
Since $\{\xi_1,\ldots,\xi_{2m}\}$ is algebraically independent in $S_m$, $\Phi_{i,m}$ is injective for $i<2m+1$.

\begin{lem} The kernel of $\Phi_{2m+1,m}$ is a principal ideal. Furthermore, the generator is of the form $aT_{2m+1}-b$ with $a,b\in R_{2m}\setminus\{0\}$.
\end{lem}
\begin{proof} 
The field of fractions of $S_m^{\orthp{2m}{q}}$ is $\field_q(\xi_1,\ldots,\xi_{2m})$. Write $\xi_{2m+1}=\overline{b}/\overline{a}$ with 
$\overline{a},\overline{b}\in\field_q[\xi_1,\ldots,\xi_{2m}]=\image(\Phi_{2m,m})$.
We may assume that $\overline{a}$ and $\overline{b}$ are relatively prime. Since $\Phi_{2m,m}$ is injective, there are unique $a,b\in R_{2m}$ with
$\Phi_{2m,m}(a)=\overline{a}$ and   $\Phi_{2m,m}(b)=\overline{b}$. Evaluating $\Phi_{2m+1,m}(aT_{2m+1}-b)$ gives $\overline{a}\xi_{2m+1}-\overline{b}=0$.
Therefore $aT_{2m+1}-b$ lies in the kernel of $\Phi_{2m+1,m}$. 
For any $F\in R_{2m+1}$, for $\ell$ sufficiently large we can divide by $aT_{2m+1}-b$ to get $a^{\ell}F=(aT_{2m+1}-b)f+r$ with $f,r\in R_{2m}$.
Applying $\Phi_{2m+1,m}$ gives $$\Phi_{2m+1,m}(a^{\ell}F)=\overline{a}^\ell\Phi_{2m+1,m}(F)=\Phi_{2m,m}(r).$$
Since $\Phi_{2m,m}$ is injective we see that $a^{\ell}F$ is in
$\ker(\Phi_{2m+1,m})$ if and only if $a^{\ell}F=(aT_{2m+1}-b)f$.
Since $\overline{a}$ and $\overline{b}$ are relatively prime, $aT_{2m+1}-b$ is prime in $R_{2m+1}$.   Hence $F\in\ker(\Phi_{2m+1,m})$ if and only if $F$ is divisible by $aT_{2m+1}-b$. 
\end{proof} 

Let $\overline{\Phi}_{i,m}$ denote the composition of $\Phi_{i,m}$ followed by projection onto $S_m/x_1S_m$. 
Let $\sigma$ denote the inclusion of $S_{m-1} \subset S_m$.
Recall that $\psi_{[m,1]}$ is the algebra homomorphism from $S_m$ to $S_m$ which takes $a$ to $a^q-ax_m^{q-1}$ for $a$ homogeneous of degree one.

\begin{lem}\label{phi_bar_lemma} (a) $\ker(\psi_{[m,1]})=x_mS_m$ and $\psi_{[m,1]}\circ\psi_{[m,1]}(f)=\psi_{[m,1]}(f)^q$.\\
(b) $\ker(\overline{\Phi}_{i,m})=
\ker(\Phi_{i,m-1})=\ker(\psi_{[m,1]}\circ\Phi_{i,m}).$\\
(c) If $f\in\ker(\overline{\Phi}_{i,m})$ then $u_m$ divides  $\Phi_{i,m}(f)$.
\end{lem}
\begin{proof} (a) Clearly $x_m S_m\subseteq \ker(\psi_{[m,1]})$. 
To see that $x_mS_m=\ker(\psi_{[m,1]})$, identify $S_m/x_mS_m$ with $\field_q[y_m,y_{m-1},\ldots, x_{m-1}]$ and observe that $\psi_{[m,1]}$ maps this ring
injectively to $\field_q[y_m^q-y_mx_m^{q-1},y_{m-1}^q-y_{m-1}x_m^{q-1},\ldots, x_{m-1}^q-x_{m-1}x_m^{q-1}]$.
For $a\in S_m$ with $\deg(a)=1$, we have $\psi_{[m,1]}\circ\psi_{[m,1]}(a)=\psi_{[m,1]}(a^q-ax_m^{q-1})=\psi_{[m,1]}(a)^q$. 
The general result follows from the fact that taking the $q^{th}$ power is an algebra monomorphism on $S_m$.

(b) If we identify $S_m/x_mS_m$ with $\sigma(S_{m-1})[y_m]$, then $\overline{\Phi}_{i,m}=\sigma\circ\Phi_{i,m-1}$. 
Since $\sigma$ is a monomorphism, we see that $\ker(\overline{\Phi}_{i,m})=\ker(\Phi_{i,m-1})$.
From (a) we have $\ker(\psi_{[m,1]})=x_mS_m$. Therefore $\ker(\overline{\Phi}_{i,m})=\ker(\psi_{[m,1]}\circ\Phi_{i,m})$ follows from the definition of $\overline{\Phi}_{i,m}$.  

(c) If $f\in\ker(\overline{\Phi}_{i,m})$ then $x_m$ divides $\Phi_{i,m}(f)$. Since $\Phi_{i,m}(f)$ is invariant, 
if $x_m$ divides $\Phi_{i,m}(f)$,
each element in the orbit of $x_m$ must divide $\Phi_{i,m}(f)$. The result follows from the fact that $u_m$ is a product of relatively prime linear factors from the orbit of $x_m$ 
\end{proof}

In the following, for $i<2m+1$, we identify $R_i$ and $\Phi_{i,m}(R_i)$. 
Recall that $\psi:S_m\to S_m[t]$ is the algebra homomorphism determined by $\psi(a)=a^q-at^{q-1}$ for $a$ homogeneous of degree one. 
Note that $\psi$ followed by evaluating $t$ at $x_m$ gives $\psi_{[m,1]}:S_m\to S_m$.
Using Corollary~\ref{com_st} and the relationship between $\psi$ and the complete Steenrod operator gives 
\begin{equation}\label{psi_equation:1}
    \psi(\xi_1)=\xi_1^q-\xi_2 t^{q-1}+\xi_1 t^{2(q-1)},
\end{equation}
\begin{equation}\label{psi_equation:2}
\psi(\xi_2)=\xi_2^q-\xi_3 t^{q-1}-2\xi_1^qt^{q(q-1)}+\xi_2 t^{(q+1)(q-1)}
\end{equation}
and, for $i\geq 3$,
\begin{equation}\label{psi_equation:3}
\psi(\xi_{i})=\xi_{i}^q-\xi_{i+1} t^{q-1}-\xi_{i-1}^qt^{q^{i-1}(q-1)}+\xi_{i} t^{(q^{i-1}+1)(q-1)}\, .
\end{equation}
Therefore $\psi$ restricts to a map from $R_{n-1}$ to $R_{n}[t]$.

For $f\in R_{n}$, define $\rdeg(f)$ to be the degree of $f$ as a polynomial in 
the variables $\xi_1,\ldots,\xi_{n}$ 
and define $\wdeg(f)$ to be the degree of $f$ as an element of $S_m$.
For example, $\wdeg(\xi_{i})=q^{i-1}+1$ while $\rdeg(\xi_{i})=1$.
We will use the weighted reverse lexicographic order on $R_{n}$ with $\wdeg$
as the weight. This means we first compare $\wdeg$ and if two monomials
have the same $\wdeg$, we use reverse lex to determine the order
(see \cite[Example 1.2.8]{greuel+pfister:2008}).
In Magma \cite{magma:97}, this order is called Graded Reverse Lexicographical (Weighted)
and is implemented using the command {\tt grevlexw}
(see the online Magma handbook).  
Using this order $\xi_{i+1}>\xi_i$ and
$\xi_2^{q+1}>\xi_1^q\xi_3>\xi_2^{q-1}\xi_1^{q+1}$
(all three monomials have $\wdeg$ equal to $1+2q+q^2$ and
larger exponents on $\xi_1$ gives smaller monomials). 

Define $\nu:R_{n}\to \Z\cup\{\infty\}$ by $\nu(0)=\infty$ and, for $f\not=0$, $\nu(f)$ is the minimum value of $\rdeg$ on the non-zero terms of $f$.
For example, $\nu(\xi_2^q)=q$ and $\nu(\xi_2^q-\xi_3\xi_2^{(q-1)/2})=(q+1)/2$.
Note that if $\nu(f)=\ell$, then $f\in(R_{n})_+^{\ell}\setminus (R_{n})_+^{\ell+1}$
where $(R_{n})_+$ denotes the ideal of 
$R_{n}$ generated by $\{\xi_{n},\xi_{n-1},\dots,\xi_1\}$.  
It is easy to verify that for $f,h\in R_{n}$, we have $\nu(fh)=\nu(f)\nu(h)$ and
$\nu(f+h)\geq{\rm min}\{\nu(f),\nu(h)\}$.
It follows from \cite[Ch IV, \S 9, Theorem 14]{zariski+samuel-commalge:75}, that 
defining $\nu(f/h)=\nu(f)-\nu(h)$ extends 
$\nu$ to a discrete valuation on $\F(R_{n})$.
Note that the associated Discrete Valuation Ring, $\{a\in\F(R_{n})\mid \nu(a)\geq 0\}$, properly contains $R_{n}$, i.e.,
$R_{n}$ is not the associated Discrete Valuation Ring.

The following is a consequence of Corollary \ref{com_st}.

\begin{lem} \label{nu_st_lem} For $f\in R_{n}$, we have $\nu(\PP^i(f))\geq \nu(f)$.
\end{lem}

\begin{lem} \label{nu_ord_lem}
Suppose $\beta$ is a monomial in $R_{n}$ with $\wdeg(\beta)=\wdeg(\xi_{i}^b\xi_{i+1}^a)$.
If $\nu(\beta)>\nu(\xi_{i}^b\xi_{i+1}^a)$ then $\lt(\beta)<\lt(\xi_{i}^b\xi_{i+1}^a)$ using the weighted revlex order.
\end{lem}
\begin{proof} It is sufficient to show that $\beta$ is in the ideal  
$\langle \xi_1,\ldots,\xi_{i-1},\xi_{i}^{b+1}\rangle \subset R_{n}$.
Suppose, by way of contradiction, that this is not the case. Since $\nu(\beta)>b+a$, $\beta$ has at least $a+b+1$ factors.  Since $\beta$ is not in the ideal 
$\langle \xi_1,\ldots,\xi_{i-1},\xi_{i}^{b+1}\rangle$, 
at least $a+1$ of these factors must be greater than $\xi_{i}$.
The smallest weighted degree of such a monomial is
$$b(q^{i-1}+1)+(a+1)(q^{i}+1)>b(q^{i-1}+1)+a(q^{i}+1)=\wdeg(\xi_{i}^b\xi_{i+1}^a), $$
contradicting the hypothesis that $\wdeg(\beta)=\wdeg(\xi_{i}^b\xi_{i+1}^a)$.
\end{proof}

Define an $m\times (m+1)$ matrix
$$M_m:=\begin{pmatrix}
\xi_{n} & \xi_{n-1} && \cdots & \xi_{m}\\
\xi_{n-1}^q & \xi_{n-2}^q && \cdots & \xi_{m-1}^q\\
&&\vdots&& \\
\xi_{m+1}^{q^{m-1}} & \xi_{m} ^{q^{m-1}} && \cdots & \xi_1^{q^{m-1}} 
\end{pmatrix}
$$ and let $M(i,m)$ denote the minor formed by removing column $i+1$ from $M_m$.
Note that $\nu(M(i,m))=1+q+\cdots +q^{m-1}$.

\begin{lem} \label{nu_min_lem} $\nu(\PP^{e(i,m)}(M(0,m))-M(i,m))>1+q+\cdots +q^{m-1}$.
\end{lem}
\begin{proof} Let $m_{j,k}$ denote the row $j$ column $k$ entry of $M_m$.
Then
$$M(0,m)=\sum_{\sigma\in\Sigma_m}\sgn(\sigma)\left(\prod_{k=1}^m m_{\sigma(k),k+1}\right)$$
and, for $0<i\leq m$,
$$M(i,m)=\sum_{\sigma\in\Sigma_m}\sgn(\sigma)\left(\prod_{k=1}^{i} m_{\sigma(k),k}\right)\left(\prod_{k=i+1}^m m_{\sigma(k),k+1}\right).$$
Note that $\nu(m_{j,k})=q^{j-1}$. For $k>1$, it follows from Lemma \ref{com_st} that $\PP^{q^{n-k}}(m_{j,k})=m_{j,k-1}$,
and $\nu(\PP^{\ell}(m_{j,k}))>q^{j-1}$ unless $\ell=0$ or $\ell=q^{n-k}$.
Therefore, since $e(i,m)=q^{n-2}+q^{n-3}+\cdots +q^{n-i-1}$, the Cartan identity yields
$$\PP^{e(i,m)}\left(\prod_{k=1}^m m_{\sigma(k),k+1}\right)
=\left(\prod_{k=1}^{i} m_{\sigma(k),k}\right)\left(\prod_{k=i+1}^m m_{\sigma(k),k+1}\right)+\delta$$
with $\nu(\delta)>1+q+\cdots +q^{m-1}$.
Hence $\nu(\PP^{e(i,m)}(M(0,m))-M(i,m))>1+q+\cdots +q^{m-1}$.
\end{proof}

Finally we are ready to formulate the main theorem.
We will show that $S_m^{G_m}$ is generated by the $3m-1=n+m-1$ invariants
$$\cS_m:=\{\xi_1,\ldots,\xi_{n-1},d_{1,m},\ldots,d_{m,m}\}.$$
The proof is by induction on $m$ and relies on the following lemma.


\begin{lem}\label{oplus_tech-lem} Suppose $m>1$ and 
$S_{m-1}^{G_{m-1}}$ is generated by 
$\cS_{m-1}$.\\
(a) $u_m\in R_{n-1}$, $\ker(\Phi_{n-1,m-1})=u_m R_{n-1}$, 
and $u_m d_{i,m}\in R_{n}$.\\
(b) $\PP^i(u_m)=0$ for $0<i<q^{m-1}$ 
and $\PP^i(u_m)/u_m\in R_{n-1}$ for $q^{m-1}\leq i<q^{n-2}$.
For $i=\ell q^{n-2}+r$ with $0\leq\ell<q$ and $0\leq r<q^{n-2}$, $u_m^{\ell+1}\PP^i(d_{j,m})\in R_{n}$ 
with $\xi_{n}$-degree at most $\ell+1$. If $\ell=q-1$ and $i\not\equiv_{(q)} 0$ then 
$u_m^q\PP^i(d_{j,m})$ has $\xi_{n}$-degree at most $q-1$.\\
(c) Interpreting $u_{m-1}$ and $u_{m-1} d_{i,m-1}$ as elements of $R_{n-2}\subset S_m^{G_m}$, we have
$$u_m=(\xi_{n-1}+c_{n-1,m})u_{m-1}^q+\sum_{i=1}^{m-1}(-1)^i(\xi_{n-1-i}+c_{n-1-i,m})(u_{m-1}d_{i,m-1})^q$$
with $c_{j,m}\in R_{j-1}$ and $u_m d_{i,m}-(u_{m-1}d_{i-1,m-1})^q\xi_{n}\in R_{n-1}$. \\
(d) Using the weighted revlex order on $R_{n}$, 
$$\lt(u_m)=(-1)^{\lfloor m/2 \rfloor}\xi_{m}^{q^{m-1}+q^{m-2}+\cdots+1}$$ and 
$\lt(u_m d_{i,m})=(-1)^{\lfloor m/2 \rfloor}\xi_{m}^{q^{m-1-i}+\cdots+1}\xi_{m+1}^{q^{m-1}+\cdots+q^{m-i}}$.
Note that for $i=m$, $\lt(u_m d_{m,m})$ is a power of $\xi_{m+1}$.\\
Furthermore $$\nu(u_m)=\nu(u_md_{i,m})=1+q+\cdots +q^{m-1},$$
$u_m=M(0,m)+\delta_{0,m}$ with $\delta_{0,m}\in R_{n-1}$,
$u_md_{i,m}=M(i,m)+\delta_{i,m}$ with $\delta_{i,m}\in R_{n}$ and, for $0\leq i\leq m$, $\nu(\delta_{i,m})>1+q+\cdots +q^{m-1}$.\\ 
(e) There exists $c_{n,m}\in R_{n-1}$ such that
 $$\xi_{n}=\left(\sum_{i=1}^{m}(-1)^{i+1}(\xi_{n-i}+c_{n-i,m})d_{i,m}\right)-c_{n,m}.$$
(f) For $1\leq i\leq m-1$ there exist $\gamma_{k,m}^{(i)}\in R_{n-1}$ with $\nu(\gamma_{k,m}^{(i)})>q^i$
such that
$$\xi_{n-i}^{q^i}=\left(\sum_{j=1}^m (-1)^{j+1}(\xi_{n-i-j}^{q^i}+\gamma_{n-i-j,m}^{(i)})d_{j,m}\right)-\gamma_{n-i,m}^{(i)}.$$
\end{lem}

We postpone the proof of Lemma~\ref{oplus_tech-lem} to 
Section~\ref{proof_of_tech-lem}.

\begin{remark}\label{um_div_rem}
It follows from Lemma~\ref{phi_bar_lemma}(b) and Lemma~\ref{oplus_tech-lem}(a)
that if $u_m$ divides $f$ in $S_m$ and $f\in R_{n-1}$ then $f/u_m\in R_{n-1}$.
To see this, suppose $u_m$ divides $f$ in $S_m$ and $f\in R_{n-1}$. 
Then $x_m$ divides $f$ in $S_m$, which means that $f\in\ker(\psi_{[m,1]})$.
Since $f\in R_{n-1}$, we have $f\in\ker(\psi_{[m,1]}\circ \Phi_{n-1,m})$.
Using  Lemma~\ref{phi_bar_lemma}(b) gives
$\ker(\Phi_{n-1,m-1})=\ker(\psi_{[m,1]}\circ \Phi_{n-1,m})$. 
From Lemma~\ref{oplus_tech-lem}(a), we have
$\ker(\Phi_{n-1,m-1})=u_m R_{n-1}$. Therefore $f\in u_mR_{n-1}$, as required.
\end{remark}

Recall from Theorem~\ref{ortho_hsop_thm} the homogeneous system of parameters
$$\HH=\{\xi_1,\ldots,\xi_{m},d_{1,m},\ldots,d_{m,m}\}$$ 
where $\deg(\xi_i)=q^{i-1}+1$ and $\deg(d_{i,m})=q^{n-1}-q^{n-1-i}$. Let 
$\B$ denote the set of monomial factors of 
$$\Gamma=\prod_{i=m+1}^{n-1}\xi_{i}^{q^{n-i}-1}= \prod_{j=1}^{m-1}\xi_{n-j}^{q^j-1}\ .$$ 

We now state our main theorem.

\begin{thm} \label{oplus_thm_v2} 
Suppose $m>1$. $S_m^{G_m}$ is generated by 
$$\cS_m=\HH\cup\{\xi_{m+1},\ldots,\xi_{n-1}\}
=\{\xi_1,\ldots,\xi_{n-1},d_{1,m},\ldots,d_{m,m}\}.$$ 
Furthermore, $S_m^{G_m}$ is the free $\field_q[\HH]$-module with basis $\B$ and is the complete intersection with the relations given in part (f) of 
Lemma~\ref{oplus_tech-lem}.
\end{thm}

\begin{proof} 
We assume the conclusions of Lemma~\ref{oplus_tech-lem} hold for $S_m^{G_m}$.
Note that $S_1^{G_1}$ is the polynomial algebra generated by $\xi_1=x_1y_1$
and $\xi_2/\xi_1=y_1^{q-1}+x_1^{q-1}$.

    Let $A$ denote the subalgebra  of $S_m^{G_m}$ generated by $\cS_m$. 
    Using Theorem~\ref{ortho_hsop_thm},
$A\subseteq S_m^{G_m}$ is an integral extension.
Using part (e) of Lemma~\ref{oplus_tech-lem}, $R_{n}\subset A$. Therefore 
$\F(A)=\F(R_{n})=\F(S_m^{G_m})$.
Thus to show $A=S_m^{G_m}$, it is sufficient to show that $A$ is integrally closed in its field of fractions.
We will use \cite[Proposition 1.1]{Kropholler+Rajaei+Segal:05} to prove that $A$ is a UFD. 
Recall that a UFD is integrally closed in its field of fractions (see, for example, \cite[Proposition 4.10]{eisenbud:95}).

Using the definition and part (f) of Lemma~\ref{oplus_tech-lem}, $A$ is the $\field_q[\HH]$-module generated by $\B$.
Note that the number of elements of the set $\B$ is $\prod_{i=1}^{m-1}q^i=q^{m(m-1)/2}$.
From Theorem~\ref{sylow_thm}, we know that $S_m^{P_m}$ is Cohen-Macaulay.
Therefore $S_m^{G_m}$ is Cohen-Macaulay
and $S_m^{G_m}$ is a free $\field_q[\HH]$-module of rank
    $$
        \frac{\prod_{i=0}^{m-1}(q^i+1)\prod_{j=1}^m(q^{n-1}-q^{n-1-j})}{2q^{m(m-1)}(q^m-1)\prod_{k=1}^{m-1}(q^{2k}-1)},
    $$
which simplifies to $q^{m(m-1)/2}$. 
Therefore  the field extension $\F(\field_q[\HH])\subset \F(S_m^{G_m})$ is of degree $q^{m(m-1)/2}$. 
Since $\F(A)=\F(S_m^{G_m})$, the degree of the field extension $\F(\field_q[\HH])\subset \F(A)$ is $q^{m(m-1)/2}$.
We know that $\B$ is a spanning set for $\F(A)$ over $\F(\field_q[\HH])$. By comparing the order of $\B$ with the degree of the extension, 
we see that $\B$ is linearly independent over  $\F(\field_q[\HH])$. Therefore $A$ is a free $\field_q[\HH]$-module with basis $\B$ and $A$ is Cohen-Macaulay. 
Since $A$ is a free $\field_q[\HH]$-module with basis $\B$, the Hilbert series for $A$ is
$$HS(A,t)=\left(\prod_{i=1}^m\frac{1}{1-t^{q^{n-1}-q^{n-1-i}}}\right)
\left(\prod_{i=1}^{m}\frac{1}{1-t^{q^{i-1}+1}}\right)
\left(\prod_{i=m+1}^{n-1}\frac{1-t^{q^{n-1}+q^{n-i}}}{1-t^{q^{i-1}+1}}\right).$$

The complete intersection with relations given by 
Lemma~\ref{oplus_tech-lem}(f) surjects onto $A$.
Since the the Hilbert series of the complete intersection coincides with the Hilbert series for $A$, the surjection is a bijection.
Hence $A$ is the complete intersection with relations given in Lemma~\ref{oplus_tech-lem}(f). Therefore, to complete the proof of the theorem, 
we need only show that $A=S_m^{G_m}$.

From Lemma~\ref{oplus_tech-lem}(d), for $m$ and $m-1$, we have 
$$\lt(u_m)=(-1)^{\lfloor m/2 \rfloor}\xi_{m}^{q^{m-1}+q^{m-2}+\cdots+1}$$ and $\lt(u_{m-1})=(-1)^{\lfloor (m-1)/2 \rfloor}\xi_{m-1}^{q^{m-2}+q^{m-3}+\cdots+1}$ (for $m=2$, $u_{m-1}=u_1=x_1y_1=\xi_1$).
Therefore, since $\HH$ is an homogeneous system of parameters, $\{u_{m-1},u_m\}$ is a partial homogeneous system of parameters. Since $A$ is Cohen-Macaulay, this means that 
$[u_{m-1},u_m]$ is a regular sequence in $A$. From 
Lemma~\ref{oplus_tech-lem}(a) for $m-1$, when $m>2$, we have $\ker(\Phi_{n-3,m-2})=u_{m-1} R_{n-3}$. Therefore $u_{m-1}$ is prime in $R_{n-3}$. If $m=2$, $u_1=\xi_1$ is prime in $R_1=\field_q[\xi_1]$.
Since $R_{n}$ is a polynomial extension of $R_{n-3}$, $u_{m-1}$ is prime in $R_{n}$ and in its localisation $R_{n}[u_m^{-1}]$.
Using part (e) of Lemma~\ref{oplus_tech-lem} , $R_{n}\subset A$ and by definition 
$u_m d_{i,m}\in R_{n}$. Therefore $A[u_m^{-1}]=R_{n}[u_m^{-1}]$ and
$u_{m-1}$ is prime in $A[u_m^{-1}]$. To satisfy the hypotheses of \cite[Proposition 1.1]{Kropholler+Rajaei+Segal:05} and complete the proof, we need to show that
$A[u_{m-1}^{-1}]$ is a UFD. 

From Lemma~\ref{oplus_tech-lem}(c), we have $u_m d_{1,m}-u_{m-1}^q\xi_{n}\in R_{n-1}$. 
Therefore $R_{n}\subset R_{n-1}[d_{1,m}, u_{m-1}^{-1}]$ and
$\F(R_{n})=\F(R_{n-1}[d_{1,m}])$. Hence $R_{n-1}[d_{1,m}]$ is a polynomial algebra and 
$R_{n-1}[d_{1,m}, u_{m-1}^{-1}]$ is a UFD.
We prove $A[u_{m-1}^{-1}]$ is a UFD by showing  $A[u_{m-1}^{-1}]=R_{n-1}[d_{1,m},u_{m-1}^{-1}]$. By definition, $R_{n-1}[d_{1,m}]\subset A$.
From Lemma~\ref{oplus_tech-lem}(c), we have $u_m d_{i,m}-(u_{m-1}d_{i-1,m-1})^q\xi_{n}\in R_{n-1}$ for $i>1$. Cross-multiplying to eliminate $\xi_{n}$ gives
$$u_{m-1}^q u_m d_{i,m}- (u_{m-1}d_{i-1,m-1})^qu_m d_{1,m}\in R_{n-1}.$$ 
Using $\ker(\Phi_{n-1,m-1})=u_m R_{n-1}$ (Lemma~\ref{oplus_tech-lem}(a)) and
Remark~\ref{um_div_rem}, we see that 
$$u_{m-1}^q d_{i,m}- (u_{m-1}d_{i-1,m-1})^q d_{1,m}\in R_{n-1}$$ (with $u_{m-1}d_{i-1,m-1}\in R_{n-2}$). Thus $d_{i,m}\in R_{n-1}[d_{1,m},u_{m-1}^{-1}]$, as required.
\end{proof}

\begin{remark}
    The rank of $\invring$ as a free $\field_q[\HH]$-module is $q^{m(m-1)/2}$.
\end{remark}

\begin{remark}
  Note that the relations in the associated graded algebra are given by
  $$\xi_{n-i}^{q^i}=\sum_{j=1}^m (-1)^{j+1}\xi_{n-i-j}^{q^i}d_{j,m} \quad
  \text{ for } i=1,2,\dots,m-1.$$
  This formula is also true for $i=0$.
\end{remark}

\begin{cor}\label{orthogonal minimality}
    $\invring$ is minimally generated by 
    $$\{\xi_1,\xi_2,\dots,\xi_{n-1},d_{1,m},d_{2,m},\dots,d_{m,m}\}\ .$$
\end{cor}

\begin{proof}
    The invariants $\xi_1,\xi_2,\dots,\xi_{n-1}$ are algebraically independent 
and of degree less than $\deg (d_{i,m})$ for all $i$.  Hence none of these $\xi_{i}$
is redundant.  But Theorem~\ref{orth_var} shows that at least $m$ other invariants 
are required to generate $\invring$.
\end{proof}

\section{Sylow Invariants} \label{sylow_section}
The order of $\orthp{2m}{q}$ is $2q^{m(m-1)}(q^m-1)\prod_{j=1}^{m-1}(q^{2j}-1)$,
see \cite[page 141]{taylor-classicalgroups:92}. 
Therefore a Sylow $p$-subgroup has order $q^{m(m-1)}$.
The field of fractions for the Sylow invariants was calculated in 
\cite{Ferreira+Fleischmann-fields:16}
 and the ring of invariants for the $m=2$ case was calculated in \cite{Ferreira+Fleischmann-rings:17}.

Let $H$ denote the {\it Hook} subgroup of $\orthp{2m}{q}$ consisting of unipotent matrices which differ from the identity only in the first row and the last column.
A direct calculation shows that for $h\in H$ we have$$y_m\cdot h= y_m+\sum_{i=1}^{m-1} (b_iy_i+a_ix_i-a_ib_ix_m),$$ $x_m\cdot h=x_m$, and for $j<m$,
  $y_j\cdot h=y_j-a_jx_m$, $x_j\cdot h=x_j-b_jx_m$ where $a_i,b_i\in\field_q$.
Hence the order of $H$ is $q^{2m-2}$. Note that $H$ is a normal subgroup of $P_m$. Identify $P_{m-1}$ with pointwise stabiliser of $\{y_m,x_m\}$ in $P_m$.
Then $P_m=P_{m-1}H$.
(Note that $(m-1)(m-2)+(2m-2)=m(m-1)$.)
Therefore $S_m^{P_m}=(S_m^H)^{P_{m-1}}$.
We will calculate $S_m^{P_m}$ by first computing the Hook invariants. 
Since the orbit product of $y_m$ over $H$ coincides with the orbit product of
 $y_m$ over $P_m$,
we have $N(y_m)=N^H(y_m)$ where $N^H(a)$ is the orbit product of $a$ over $H$.

\subsection{Computing the Hook Invariants} \label{HgroupSec}
Consider the subalgebra $$Q:=\field_q[y_{m-1},\ldots,y_1,x_1,\ldots,x_m]\subset S_m.$$
For $i<m$, define $Y_i:=y_i^q-y_ix_m^{q-1}$ and $X_i=x_i^q-x_ix_m^{q-1}$.
It is easy to see that $$Q^H:=\field_q[x_m][Y_i,X_i\mid i=1,\ldots,m-1].$$
Since $\xi_1$ has degree $1$ in $y_m$, we have $S_m^H[x_m^{-1}]=Q^H[\xi_1,x_m^{-1}]$
(see \cite[Theorem 2.4]{campbell+chuai-local:07}).

Using the lex order $\lt(\xi_j^q)=x_m^qy_m^{q^j}=\lt(\xi_{j+1}x_m^{q-1})$.
To subduct the \tat\ $\xi_j^q-\xi_{j+1}x_m^{q-1}$, we first eliminate $y_m$.
For $j=1$, we get
\begin{equation}\label{ortat:1}
  \xi_1^q-\xi_2x_m^{q-1}+\xi_1x_m^{2q-2}=\psi_{[m,1]}(\xi_1)\in Q^H
  \end{equation}
  (compare with Equation~\ref{psi_equation:1}).
  For $j>1$, we have
$$\xi_{j}^q-\xi_{j+1}x_m^{q-1}-\xi_1^qx_m^{q^j-q}+\xi_1x_m^{q^j+q-2}\in Q^H.$$
Since $Q^H$ is a polynomial algebra whose generators have relatively prime lead terms,
the subduction algorithm can be used to write an element of $Q^H$ as a polynomial in the generators. Therefore the \tat\ $\xi_{j}^q-\xi_{j+1}x_m^{q-1}$ subducts to zero in
$Q^H[\xi_1,\ldots,\xi_{j+1}]$. For $j=1$, the subduction is given by Equation~\ref{ortat:1}. For $j>1$, the subduction is given by the following lemma.

\begin{lem} \label{ortat:2} For $j>1$, we have
$$\xi_{j}^q-\xi_{j+1}x_m^{q-1}=\xi_1^qx_m^{q^j-q}-\xi_1x_m^{q^j+q-2}
+\sum_{i=1}^j\psi_{[m,1]}(\xi_i)x_m^{q^j-q^i}.$$
\end{lem}
\begin{proof} The proof is by induction on $j$.
Using Equation~\ref{psi_equation:2}
 and the fact that $\psi$ followed by evaluating $t$ at $x_m$ gives $\psi_{[m,1]}$, 
 we have
 \begin{eqnarray*} 
 \xi_{2}^q-\xi_{3}x_m^{q-1}
 &=&2\xi_1^qx_m^{q^2-q} -\xi_2x_m^{q^{2}-1}+\psi_{[m,1]}(\xi_{2})\cr
 &=&\psi_{[m,1]}(\xi_{2})+\xi_1^qx_m^{q^2-q}
 +x_m^{q^2-q}(\xi_1^q-\xi_2 x_m^{q-1}).
 \end{eqnarray*}
 Using Equation~\ref{ortat:1} to substitute for $\xi_1^q-\xi_2 x_m^{q-1}$
 gives
 \begin{eqnarray*} \xi_{2}^q-\xi_{3}x_m^{q-1}&=&\psi_{[m,1]}(\xi_{2})+\xi_1^qx_m^{q^2-q}
 +x_m^{q^2-q}(\psi_{[m,1]}(\xi_1)-\xi_1x_m^{2q-2})\cr
 &=&\xi_1^qx_m^{q^2-q}-\xi_1 x_m^{q^2+q-2}+\psi_{[m,1]}(\xi_{2})
 +x_m^{q^2-q}\psi_{[m,1]}(\xi_1),
 \end{eqnarray*} 
 proving the result for $j=2$.

For $j>2$, Equation~\ref{psi_equation:3} gives
\begin{eqnarray*} 
 \xi_{j}^q-\xi_{j+1}x_m^{q-1}&=&\xi_{j-1}^qx_m^{q^{j-1}(q-1)}
 -\xi_jx_m^{(q^{j-1}+1)(q-1)}+\psi_{[m,1]}(\xi_{j})\cr
 &=&\psi_{[m,1]}(\xi_{j})
 +x_m^{q^{j-1}(q-1)}(\xi_{j-1}^q-\xi_jx_m^{q-1}).
 \end{eqnarray*}  
 Using the induction hypothesis
 \begin{eqnarray*}
 x_m^{q^{j-1}(q-1)}(\xi_{j-1}^q-\xi_jx_m^{q-1})
 &=&x_m^{q^{j-1}(q-1)}(\xi_1^q x_m^{q^{j-1}-q}
 -\xi_1x_m^{q^{j-1}+q-2} \cr
 &&+\sum_{i=1}^{j-1}\psi_{[m,1]}(\xi_i)x_m^{q^{j-1}-q^i})\cr
&=&\xi_1^qx_m^{q^j-q}-\xi_1x_m^{q^j+q-2}
+\sum_{i=1}^{j-1}\psi_{[m,1]}(\xi_i)x_m^{q^j-q^i}.
\end{eqnarray*}
Therefore
$$\xi_{j}^q-\xi_{j+1}x_m^{q-1}=\xi_1^qx_m^{q^j-1}-\xi_1x_m^{q^j+q-2}
+\sum_{i=1}^j\psi_{[m,1]}(\xi_i)x_m^{q^j-q^i},$$
as required.
\end{proof}

\begin{lem} \label{ortat:3} $x_mN(y_m)-\xi_{n-1}\in Q^H[\xi_{1},\ldots,\xi_{n-2}]$.
  \end{lem}
We postpone the proof of the lemma to the end of the section. 
    
  Define $$A:=Q^H[\xi_{1},\ldots,\xi_{n-2}][N(y_m)]$$
  and
  $\mathcal{W}:=\{X_i,Y_i\mid i=1,\ldots,m-1\}$.
\begin{lem} \label{orKbasis} The set $$\mathcal{W}\cup\{N(y_m), x_m,\xi_{1},\dots,\xi_{n-2}\}$$
is a Khovanskii basis for $A$ using the lexicographic order.
Furthermore, $A$ is a complete intersection with relations given by the subductions of the \tat s
$\xi_{n-2}^q-x_m^qN(y_m)$ and $\xi_{j}^q-x_m^{q-1}\xi_{j+1}$ for $j=1,\ldots,n-3$.
\end{lem}
\begin{proof}
The \tat s $\xi_{j}^q-x_m^{q-1}\xi_{j+1}$ for $j=1,\ldots, n-3$ subduct to zero using Equation \ref{ortat:1} and Lemma~\ref{ortat:2}.
The subduction of the \tat\ $x_m^qN(y_m)-\xi_{n-2}^q$ is constructed using 
Lemma \ref{ortat:3} and Lemma~\ref{ortat:2} with $j= n-2$. 
Using Lemma~\ref{ortat:3}, define 
$F:=x_mN(y_m)-\xi_{n-1}\in Q^H[\xi_{1},\ldots,\xi_{n-2}]$
so $$\xi_{n-2}^q-x_m^qN(y_m)=\xi_{n-2}^q-x_m^{q-1}\xi_{n-1}-Fx_m^{q-1}.$$
Using Lemma~\ref{ortat:2} with $j=n-2$ gives
$$\xi_{n-2}^q-x_m^qN(y_m)=-Fx_m^{q-1}+\xi_1^qx_m^{q^{n-2}-q}-\xi_1x_m^{q^{n-2}+q-2}
+\sum_{i=1}^{n-2}\psi_{[m,1]}(\xi_i)x_m^{q^{n-2}-q^i}$$
with $\psi_{[m,1]}(\xi_i)\in Q^H$.
Subduction can be used to write
$F$ in terms of the generators of  $Q^H[\xi_{1},\ldots,\xi_{n-2}]$,
completing the calculation. 

Since the non-trivial \tat s subduct to zero, we have a Khovanskii basis.
When we have a Khovanskii basis for an algebra, the ideal of relations is generated by the subductions of the non-trivial \tat s. Therefore $A$ is a complete intersection.
\end{proof}

\begin{thm} \label{orHcal} $S_m^H=A$.
\end{thm}
\begin{proof} Since $Q^H[N(y_m)]$ is generated by a homogeneous system of parameters and  $S_m^H[x_m^{-1}]=Q^H[\xi_1,x_m^{-1}]$,
it is sufficient to prove that the ring $A$ is integrally closed in its field of fractions.
 We do this by proving $x_m$ is prime in $A$ which is sufficient by 
 \cite[Theorem 20.2]{matsumura-commringtheo:86}.

It follows from Lemma~\ref{orKbasis} that $A$ is Cohen-Macaulay and has rank $q^{n-2}$ over
$Q^H[N(y_m)]$ with a basis given by the monomial factors of 
$(\xi_1\cdots\xi_{n-2})^{q-1}$.
We will show that $x_m$ is prime in $A$ by proving that $A/x_mA$ is an integral domain.

Recall that $\psi_{[m,1]}:S_m\to S_m$  is the algebra homomorphism determined by 
$\psi_{[m,1]}(a)=a^q-ax_m^{q-1}$ for $a$ homogeneous of degree one
(see Section~\ref{setting_section}).
Using Lemma~\ref{phi_bar_lemma},
 $x_mS_m= \ker(\psi_{[m,1]})$.
Let $\overline{\psi_{[m,1]}}$ denote the algebra map from $A/x_mA$ to $S_m$ induced by $\psi_{[m,1]}$. 
If $\overline{\psi_{[m,1]}}$ is injective, then $A/x_mA$ is an integral domain, proving $x_mA$ is a prime ideal and $A$ is normal.

  Define 
  $\overline{Q}:=\field_q[N(y_m),Y_{m-1},\ldots,Y_1,X_1,\ldots,X_{m-1}]$ (the subalgebra generated by $\{N(y_m)\}\cup\mathcal{W}$).
  Let $\BH$ denote the set of monomial factors of $(\xi_1\cdots\xi_{n-2})^{q-1}$.
  Using the $Q^H$-module structure of $A$, we see that
  $A/x_mA$ is a free $\overline{Q}$-module with a basis  $\BH$.
  By identifying $\overline{Q}$ with the submodule of $A/x_mA$ generated by $1$ we see that $A/x_mA$ is a complete intersection with relations
  $\xi_{i}^q =\psi_{[m,1]}(\xi_{i})$
  for $i=1,\ldots, n-2 $. From Lemma~\ref{phi_bar_lemma},
 for $f\in S_m$, $\psi_{[m,1]}\circ\psi_{[m,1]}(f)=\psi_{[m,1]}(f)^q$.
  Therefore $\psi_{[m,1]}(A)$ is a module over $\field_q[\psi_{[m,1]}(N(y_m)),Y_{m-1}^q,\dots ,X_{m-1}^q]$ with module generators $\{\psi_{[m,1]}(\beta)\mid \beta\in\BH\}$. 
  To prove that $\overline{\psi_{[m,1]}}$ is injective, it is sufficient to show that $\overline{\psi_{[m,1]}}(A/x_mA)=\psi_{[m,1]}(A)$ is a free module over 
   $$\psi_{[m,1]}(\overline{Q})=\field_q[\psi_{[m,1]}(N(y_m)),Y_{m-1}^q,\dots, X_{m-1}^q]$$ 
   with module generators
   $\{\psi_{[m,1]}(\beta)\mid \beta\in\BH\}$. 
  Note that for $\beta\in\BH$, we have $\psi_{[m,1]}(\beta)\in\field_q[Y_{m-1}^q,\dots, X_{m-1}^q]$.
  Therefore, it is sufficient to show that the 
  module generated by $\{\psi_{[m,1]}(\beta)\mid \beta\in\BH\}$
  over $\field_q[Y_{m-1}^q,\dots, X_{m-1}^q]$ is free. 
This is equivalent to showing the  
  $\field_q[y_{m-1}^q,\dots, x_{m-1}^q]$-module
  generated by $\{\xi_1^{e_1}\cdots\xi_{n-2}^{e_{n-2}}\mid e_i<q\}\subset S_{m-1}$
  is a free module of rank $q^{n-2}$.  This follows from the next lemma.
\end{proof}

\begin{lem}
    Suppose $G$ is a subgroup of $\gl{n}{q}$ and
    the invariant field $\F(S_m^G)$ is generated by
    $\{f_1,\ldots,f_{n}\}\subset S_m^G$.
    Then the $\field_q[y_{m}^q,\dots,y_1^q,x_1^q,\ldots, x_{m}^q]$-module
  generated by $\{f_1^{e_1}\cdots f_{n}^{e_{n}}\mid e_i<q\}\subset S_{m}$
  is a free module of rank $q^{n}$.
\end{lem}

\begin{proof}  
  Since the extension $\field_q(f_1,\ldots,f_{n})=\F(S_m)^G\subset \F(S_m)$ is a Galois extension,
  for any intermediate extension $\F(S_m)^G\subset  E\subset\F(S_m)$, we have $E=\F(S_m)^L$  
  for some subgroup $L\leq G$. 
 Consider $E=\F(S_m)^G[y_m^q,...,x_m^q]$. 
Since the pointwise stabiliser of $\{y_m^q,...,x_m^q\}$ in 
$\gl{n}{q}$ is $\{1\}$,
we have $\F(S_m)^G[y_m^q,...,x_m^q]=\F(S_m)$.
On the other hand, $\field_q(y_m^q,...,x_m^q)\subset \F(S_m)$ is a field extension of degree $q^n$.
Construct this extension by iteratively adjoining the $f_{i}$. 
Since $f_{i}$ is a root of $t^q-f_{i}^q\in \field_q(y_m^q,...,x_m^q)[t]$, at each step the extension has degree at most $q$. 
To complete the construction in $n$ steps, we need each step to be an extension of degree $q$.
We can adjoin the $f_{i}$ in any order. 
Therefore $t^q-f_{i}^q$ is irreducible over $\field_q(y_m^q,...,x_m^q)[f_{j} \mid j\not=i]$.
 Hence $\{f_1^{e_1}\cdots f_{n}^{e_{n}}\mid e_i<q\}$ is a basis for $\F(S_m)$ as a vector space over
  $\field_q(y_m^q,...,x_m^q)$. Thus
  $\{f_1^{e_1}\cdots f_{n}^{e_{n}}\mid e_i<q\}$ is linearly independent over $\field_q(y_m^q,...,x_m^q)$
 and generates a free  $\field_q[y_m^q,...,x_m^q]$-module of rank $q^n$.
 \end{proof}

The rest of Section \ref{HgroupSec} is devoted to the proof of 
Lemma \ref{ortat:3}.
For a non-zero element $h \in S=S_m$ we denote by $\nu_m(h)$ the greatest integer $n$ such that $h \in (x_m)^n S$, i.e., $x_m^{\nu_m(h)}$ divides $h$ in $S$
but $x_m^{\nu_m(h)+1}$ does not.  
Let $d$ be a non-negative integer.  We write $||d||_q=||d||$ to denote the digit sum in base $q$,  i.e., if $d=\sum_{i=0}^\ell d_i q^i$ with $0 \leq d_i < q$ for all $i$
 then $||d||_q = \sum_{i=0}^\ell d_i$. 
 We will say that a polynomial $f= f(y_m)=\sum_k r_k y_m^k \in S$ with each $r_k \in Q$ is {\it $(y_m,x_m)$-compliant} 
 if $\nu_m(r_k) \geq ||k|| - 1$ for all $k \geq 1$. We say that $f$ is {\it strongly compliant} if $\nu_m(r_k) \geq ||k||$ for all $k \geq 1$.
 Observe that $\xi_{i}$ is strongly compliant for all $i$ and that the set of strongly compliant polynomials is closed under addition.
 Furthermore, since $||k+\ell||\leq ||k||+||\ell||$,  the set of strongly compliant polynomials is closed under multiplication.

 \begin{lem} \label{orstcomp} If $f\in S^H$ is strongly compliant and $\deg_{y_m}(f)<q^{n-2}$, then $f\in Q^H[\xi_{1},\xi_2,\ldots,\xi_{n-2}]$.
 \end{lem}
  \begin{proof} 
  The proof is by induction on $\ell=\deg_{y_m}(f)$. For $\ell=0$, $f=r_0\in Q^H$.
  Suppose the result is true for all strongly compliant invariants with $y_m$-degree $d$ and $0\leq d<\ell<q^{n-2}$. 
  
  Consider $f=\sum_{k=0}^{\ell}r_k y_m^k \in S^H$ with $r_k\in Q$ and $r_{\ell}\not=0$. 
Since $y_m-y_m\cdot g \in Q$ for all $g \in H$, it follows that $r_{\ell}\in Q^H$. 
  Write $\ell=\sum_{i=0}^{n-3}\ell_iq^i$ with $0\leq \ell_i<q$. Define $\beta:=\prod_{j=0}^{n-3}\xi_{j+1}^{\ell_j}$. 
  Then $\lt(\beta)=y_m^{\ell}x_m^{||\ell||}$ using the lex order.
  Since $f$ is strongly compliant, $x_m^{||\ell||}$ divides $r_{\ell}$. Define $h:=f-r_{\ell}x_m^{-||\ell||}\beta$. Observe that $h\in S^H$.
  Since $\beta$ is strongly compliant, $h$ is strongly compliant with $\deg_{y_m}(h)<\ell$.
  Therefore, using the induction hypothesis, $h\in Q^H[\xi_{1},\xi_2,\ldots,\xi_{n-2}]$. Thus $f=h+r_{\ell}x_m^{-||\ell||}\beta\in Q^H[\xi_{1},\xi_2,\ldots,\xi_{n-2}]$.
   \end{proof}

 We will prove by induction on $m$ that $N(y_m)$ is $(y_m,x_m)$-compliant. As a consequence,  $x_mN(y_m)-\xi_{n-1}$ is strongly compliant.
 Lemma \ref{ortat:3} then follows from Lemma \ref{orstcomp}.
For $m=1$, we have $N(y_1)=y_1$, which is $(y_1,x_1)$-compliant.
 
 The following two lemmas are used in the proof of the induction step.

\begin{lem}\label{divisible by 9}
  Suppose that $q-1$ divides $k$, then $q-1$ divides $||k||_q$.
\end{lem}
\begin{proof}
  Write $k = \sum_s k_s q^s$ where $0 \leq k_s \leq q-1$ for all $s$.
  Since $q^s \equiv 1 \pmod{(q-1)}$ for all $s$ we have
  $$0 \equiv k = \sum_s k_s q^s \equiv \sum_s k_s = ||k||_q \pmod{(q-1)}\ 
  $$ 
  as required.
\end{proof}

Let $R'=\field_q[y_{m-1},y_{m-2},\dots,y_1,x_2,x_3,\dots,x_m]$.  Then
$S=R'[y_m,x_1]$.
Define $T:=S[t]$.   We extend the definition of $(y_m,x_m)$-compliance to $T$ by declaring that
$f=\sum_k r_k y_m^k$ with each
$r_k \in R'[x_1,t]$ is $(y_m,x_m)$-compliant
if $\nu_m(r_k) \geq ||k||-1$ for all $k \geq 1$.

Given $f(y_m,x_1)\in T$, we define
 $$\T_t(f) :=\prod_{a \in \field_q} f(y_m+a t,x_1-a x_m)\in T\ .$$

  \begin{lem}\label{orthcompl}
    If $f$ is $(y_m,x_m)$-compliant then $\T_t(f)$ is $(y_m,x_m)$-compliant.
  \end{lem}
\begin{proof}
   We work in the ring $\widetilde{T} := T[\alpha_1,\alpha_2,\dots,\alpha_q]$ where $\alpha_1,\alpha_2,\dots,\alpha_q$ are $q$ new indeterminates.
   Write $f=\sum_k \sum_j r_{k,j} y_m^k x_1^j$.  We have $\nu_m(r_{k,j}) \geq ||k||-1$ for all $k$ and $j$.
   Consider
   \begin{align*}
   \widetilde{f} &=\prod_{s=1}^q  f(y_m+\alpha_s t, x_1- \alpha_s  x_m)\in \widetilde{T}\\
    &= \prod_{s=1}^q \left(\sum_k \sum_j r_{k,j}(y_m+\alpha_s t)^k (x_1-\alpha_s  x_m)^j \right)\\
    &= \prod_{s=1}^q  \left(\sum_k \sum_j r_{k,j}\bigg(\sum_{i=0}^k \binomial{k}{i} y_m^{k-i} (\alpha_s t)^i \bigg)\bigg(\sum_{\ell=0}^j \binomial{j}{\ell} x_1^{j-\ell} (- \alpha_s x_m)^\ell \bigg)\right)
   \end{align*}
   Thus $\widetilde{f}$ is a linear combination of expressions of the form
   $$\prod_{s=1}^q \binomial{k_s}{i_s} r_{k_s,j_s} y_m^{k_s-i_s}(\alpha_s t)^{i_s}  \binomial{j_s}{\ell_s} x_1^{j_s-\ell_s}(-\alpha_s  x_m)^{\ell_s}$$
   which can be written as
 $$  \left(\prod_{s=1}^q \binomial{k_s}{i_s} r_{k_s,j_s} (\alpha_s)^{i_s}  \binomial{j_s}{\ell_s} (-\alpha_s)^{\ell_s} \right) y_m^k t^i x_1^j x_m^\ell
   $$ where $k=\sum_{s=1}^q (k_s-i_s)$, $i=\sum_{s=1}^q i_s$, $j=\sum_{s=1}^q (j_s-\ell_s)$ and $\ell=\sum_{s=1}^q \ell_s$.  

   Clearly $\widetilde{f}$ is invariant under any permutation of the $\alpha_s$.   Thus we may express 
   $\widetilde{f}$ using the elementary symmetric functions $E_i(\alpha_1,\alpha_2,\dots,\alpha_q)$ of degree $i$ 
   in the $\alpha_i$, i.e.,
   $\widetilde{f}=\sum_{\vec{e}} h_{(e_1,e_2,\dots,e_q)} E_1^{e_1} E_2^{e_2} \cdots  E_q^{e_q}$ with
   $h_{(e_1,e_2,\dots,e_q)} \in T$.  When we specialize the elementary symmetric functions $E_i(\alpha_1,\alpha_2,\dots,\alpha_q)$ to the values $E_i(\field_q)$, $\widetilde{f}$ specializes to $\T_t(f)$.
   
   Note that $E_i(\field_q)=0$ for $i \neq q-1$ and $i \geq 1$.  This follows from  
   $$X^q-X = \prod_{a\in \field_q} (X-a)= \sum_{i=0}^q (-1)^{i} E_i(\field_q) X^{q-i}.$$
Therefore in the expression for $\widetilde{f}$ it suffices to consider only terms of the form
   $h_{(0,0,\dots,0,e,0)}E_{q-1}^e$.  
   
Terms occurring in $h_{(0,0,\dots,0,e,0)}$ satisfy $i+\ell=e(q-1)$.  The term $h_{(0,0,\dots,0,0,0)}$   
   when $i=\ell=0$ is just $h_{(0,0,\dots,0,0,0)}=\prod_{s=1}^q  f(y_m, x_1) = f(y_m,x_1)^q$.
     Since $||kq||=||k||$ it is clear that 
     $h_{(0,0,\dots,0,0,0)}$ is $(y_m,x_m)$-compliant.

   We may assume that the term under consideration is non-zero which implies $\prod_{s=1}^q \binomial{k_s}{i_s}\binomial{j_s}{\ell_s}\neq 0$
   which in particular implies that $||k_s|| = ||k_s-i_s|| + ||i_s||$ for all $s$. 
       This fact known
    to Kummer \cite[pp.~115--116]{kummer:1852} in 1852 follows easily from Lucas' Lemma.
Thus $$||k|| = ||\sum_{s=1}^q (k_s-i_s)|| \leq  \sum_{s=1}^q ||k_s-i_s|| = \sum_{s=1}^q (||k_s|| - ||i_s||)\ .$$

For $h_{(0,0,\dots,0,e,0)}$ with $e\geq 1$ we have $i+\ell=e(q-1)$ is a positive multiple of $q-1$.
Thus $||i+\ell||$ is also a positive multiple of $q-1$ by Lemma~\ref{divisible by 9}. 
 In particular $||i+\ell|| \geq q-1$.

   Therefore
   \begin{align*}
   \ell+\sum_{s=1}^q ||k_s|| &\geq \ell+ ||k|| + \sum_{s=1}^q ||i_s||
     \geq \ell + ||k|| + ||\sum_{s=1}^q i_s||
     = \ell + ||k|| +||i|| \\
    &\geq ||\ell|| + ||k|| +|| i ||
    \geq ||k|| + || i+\ell ||
        \geq ||k|| + (q-1).
   \end{align*}
Thus $\nu_m(t^i x_1^j x_m^\ell\prod_{s=1}^q r_{k_s,j_s}) \geq \ell+\sum_{s=1}^q (||k_s||-1) \geq ||k||-1$ as required.
\end{proof}

Let $\overline{H}$ denote the pointwise stabiliser of $\{f_1,e_1\}$ in $H=H_m$ and observe that $\overline{H}$ also stabilises $\{y_1,x_1\}$.
 Using a relabeling of the variables which takes $y_i$ to $y_{i+1}$and $x_i$ to $x_{i+1}$, we can identify $H_{m-1}$ with
$\overline{H}$.
By induction $N^{H_{m-1}}(y_{m-1})$ is $(y_{m-1},x_{m-1})$-compliant.
Using the identification of $H_{m-1}$ and $\overline{H}$, this 
implies that $N^{\overline{H}}(y_m)$ is $(y_m,x_m)$-compliant
(we use $N^{\overline{H}}(a)$ to denote the orbit product of $a$ over $\overline{H}$).

It is easy to see that $\overline{H}$ is a normal subgroup of $H$.
Furthermore, we can identify  $H/\overline{H}$ with the subgroup of $H$ consisting of elements which stabilise all variables except $y_1$, $y_m$, and $x_1$.
Observe that $$N(y_m)=\prod_{a,b\in\field_q}N^{\overline{H}}(y_m+by_1+ax_1-abx_m)$$

Since $N^{\overline{H}}(y_m)$ is $(y_m,x_m)$-compliant
and $\deg_{x_1}(N^{\overline{H}}(y_m))=0$,
we can apply Lemma \ref{orthcompl}  to $N^{\overline{H}}(y_m)$
with $t=x_1$ to show that
$$M(y_m,x_1) := \T_{x_1}(N^{\overline{H}}(y_m))=\prod_{a\in\field_q} N^{\overline{H}}(y_m+ax_1)$$
is $(y_m,x_m)$-compliant. Applying Lemma \ref{orthcompl} to $M(y_m,x_1)$ with $t=y_1$, we see that
 \begin{align*}
\T_{y_1}(M) &= \prod_{b \in \field_q} M(y_m + b y_1,x_1-bx_m)\\
&= \prod_{b \in \field_q} \bigg(\prod_{a\in \field_q} N^{\overline{H}}(y_m+b y_1+a (x_1-bx_m)\bigg)\\
&  = \prod_{b \in \field_q} \prod_{a\in \field_q} N^{\overline{H}}(y_m+b y_1+a x_1 - ab x_m)\\
 & = N(y_m)
   \end{align*}
  is $(y_m,x_m)$-compliant. 
  
  Since $N(y_m)$ is $(y_m,x_m)$-compliant, $x_m N(y_m)$ is strongly $(y_m,x_m)$-compliant.
Thus $x_m N(y_m)-\xi_{n-1}$ is strongly $(y_m,x_m)$-compliant.
Since $\deg_{y_m}(x_m N(y_m)-\xi_{n-1})<q^{n-2}$, using Lemma \ref{orstcomp},
 we see that $$x_m N(y_m)-\xi_{n-1}\in Q^H[\xi_1,\dots,\xi_{n-2}].$$ This completes the proof of Lemma \ref{ortat:3}.
Proving  Lemma \ref{ortat:3} completes the proof of Lemma \ref{orKbasis} which in turn completes the proof of Theorem \ref{orHcal}.

\subsection{Computing the Sylow Invariants}
Recall that $H$ is a normal subgroup of $P_m$.
Furthermore, identifying $P_{m-1}$ with pointwise stabiliser of 
$\{x_m,y_m\}$ in $P_m$, we have $P_m=P_{m-1}H$.
Therefore
$S_m^{P_m}=(S_m^H)^{P_{m-1}}$. 
Let $\BH$ denote the monomial factors of $\prod_{i=1}^{n-2}\xi_{i}^{q-1}$.
Using the results of Section \ref{HgroupSec}, we see that $S_m^H$ is the free module over
$Q^H[N(y_m)]$ with basis given by the elements of $\BH$.
 Thus $$S_m^H=\bigoplus_{\gamma\in\BH} Q^H[N(y_m)]\gamma.$$
 Since the elements of $\BH$ are $P_m$-invariant, we have
 $$S_m^{P_m}=(S_m^H)^{P_{m-1}}=\bigoplus_{\gamma\in\BH }(Q^H[N(y_m)])^{P_{m-1}}\gamma.$$
 Write $\overline{S}$ for the subalgebra of $Q^H$ constructed by omitting the generator $x_m$.
 Then $Q^H[N(y_m)]=\overline{S}\otimes \field_q[x_m,N(y_m)]$ and
  $$(Q^H[N(y_m)])^{P_{m-1}}=\overline{S}^{P_{m-1}}\otimes \field_q[x_m,N(y_m)].$$
  The action of $P_{m-1}$ on $\overline{S}$ is equivalent to the action of  $P_{m-1}$ on $S_{m-1}$.
  This gives  an algebra isomorphism from $S_{m-1}^{P_{m-1}}$ to $\overline{S}^{P_{m-1}}$ which takes elements of degree $d$ to elements of degree $qd$.
  This isomorphism is the restriction of $\psi_{[m,1]}$ to 
  $S_{m-1}^{P_{m-1}}\subset S_m$.

\begin{thm}\label{sylow_thm}
  $S_m^{P_m}$ is Cohen-Macaulay and is generated by the orbit products of the variables and the $\xi_{i}$.
  \end{thm}
 \begin{proof}
 The proof is by induction on $m$. For $m=2$, we have $P_m=H$ and the result follows from Section \ref{HgroupSec}.

For the induction step, first observe that $S_m^{P_m}$ is generated by $x_m$, $N(y_m)$, elements of $\BH$ and generators for $\overline{S}^{P_{m-1}}$.
 The isomorphism from $S_{m-1}^{P_{m-1}}$ to $\overline{S}^{P_{m-1}}$ takes orbit products to orbit products and $\xi_{i}$ to
 $\psi_{[m,1]}(\xi_{i})\in\field_q[x_m,\xi_{j}\mid j\leq i+1]$.
 Thus $S_m^{P_m}$ is generated by orbit products and the $\xi_{i}$. By induction, $S_{m-1}^{P_{m-1}}$ is a free module
 over the algebra generated by the orbit products of the variables. Let $\mathcal{C}$ denote the image of the module generators in $\overline{S}^{P_{m-1}}$.
 Then $\BH\mathcal{C}$ is the set of module generators for  $S_m^{P_m}$ over the algebra generated by the orbit products of the variables and $S_m^{P_m}$ is Cohen-Macaulay.
 \end{proof}
 
Recall from Section \ref{setting_section} that $\psi_{[m,j]}:S_m\to S_m$ is the algebra map defined by
$$\psi_{[m,j]}(a):=\prod_{u\in U_j}(a-u)$$
for $\deg(a)=1$, where $U_j = \Span_{\field_q}\{x_{m-j+1},\ldots,x_m \}$.

\begin{lem}\label{lt of elements}
In the lexicographic term order
    $\lt(\psi_{[m,j]}(\xi_{i+1})) = y_{m-j}^{q^{i+j}} x_{m-j}^{q^j}$.
\end{lem}
\begin{proof}
    The linear forms $x_m,x_{m-1},\dots,x_{m-j+1}$ lie in the kernel
of $\psi_{[m,j]}$.  Furthermore for $k \leq m-j$ and $a,b$ positive integers $\lt(\psi_{[m,j]}(y_k^a x_k^b )) = y_k^{q^ja} x_k^{q^j b}$.  From these observations the result follows.
\end{proof}

Let $\sigma$ denote the inclusion of $S_{m-1}$ into $S_m$.
Define $\varphi:=\psi_{[m,1]}\circ\sigma$. The restriction of $\varphi$ gives the algebra isomorphism from $S_{m-1}^{P_{m-1}}$ to $\overline{S}^{{P_{m-1}}}\subset S_m$.

\begin{lem} \label{monomorphism_lem} 
$\varphi\circ\psi_{[m-1,j]}=\psi_{[m,j+1]}\circ\sigma$ and 
$\varphi(\psi_{[m-1,j]}(\xi_i))=\psi_{[m,j+1]}(\xi_i)$.
\end{lem}
\begin{proof} By definition, $\varphi=\psi_{[m,1]}\circ \sigma$.
Since $x_m$ is in the kernel of $\psi_{[m,1]}$,
we have $\varphi(\xi_i)=\psi_{[m,1]}(\xi_i)$. 
This proves the result for $j=0$, using the convention that $\psi_{[m-1,0]}$
is the identity map.

Suppose $j>0$. Both $\varphi\circ\psi_{[m-1,j]}$ and 
$\psi_{[m,j+1]}\circ\sigma$ are algebra maps from $S_{m-1}$ to $S_m$
so to prove the first equation it is sufficient to show that 
$\varphi(\psi_{[m-1,j]}(a))=\psi_{[m,j+1]}(\sigma(a))$ when 
$a$ is homogeneous of degree one.
Define 
$$\widetilde{U}_j:=\Span_{\field_q}\{x_{m-j},\ldots,x_{m-1} \}\ .$$
Then
$$\varphi(\psi_{[m-1,j]}(a))
=\psi_{[m,1]}\left(\sigma\left(\prod_{u\in \widetilde{U}_j}(a-u)\right)\right)
=\psi_{[m,1]}\left(\prod_{u\in \widetilde{U}_j}(\sigma(a)-u)\right).$$
Since $U_{j+1}=U_1\oplus \widetilde{U}_j$,
we have 
$$\psi_{[m,1]}\left(\prod_{u\in \widetilde{U}_j}(\sigma(a)-u)\right)
=\prod_{u\in U_{j+1}}(\sigma(a)-u)=\psi_{[m,j+1]}(\sigma(a))$$
as required. For the second equation, since $x_m$ is in the kernel of $\psi_{[m,j+1]}$, we have 
$\psi_{[m,j+1]}(\xi_i)=\psi_{[m,j+1]}(\sigma(\xi_i))
=\varphi(\psi_{[m-1,j]}(\xi_i))$. 
\end{proof}

\begin{example} In this example we consider the Sylow invariants for $\orthp{6}{q}$. It follows from Theorem \ref{sylow_thm} that $S_3^{P_3}$ is generated by
the homogeneous system of parameters $\HH_0:=\{N(y_3), N(y_2),N(y_1), N(x_1), N(x_2),x_3\}$ and $\{\xi_1,\xi_2,\xi_3,\xi_4\}$.
Using the lexicographic order the lead monomials are $$y_3^{q^4},y_2^{q^3},y_1^{q^2}, x_1^{q^2}, x_2^q, x_3, y_3x_3,y_3^qx_3, y_3^{q^2}x_3, y_3^{q^3}x_3.$$
The initial \tat s are $$\xi_1^q-x_3^{q-1}\xi_2, \, \xi_2^q-x_3^{q-1}\xi_3, \,\xi_3^q-x_3^{q-1}\xi_4, \,\xi_4^q-x_3^qN(y_3).$$
Subducting $\xi_1^q-x_3^{q-1}\xi_2$ using Equation~\ref{ortat:1}
gives $\psi_{[3,1]}(\xi_1)$, which has lead monomial $y_2^qx_2^q$. 
Adding $\psi_{[3,1]}(\xi_1)$ to the generating set and
subducting $\xi_2^q-x_3^{q-1}\xi_3$ using Lemma~\ref{ortat:2} gives $\psi_{[3,1]}(\xi_2)$, which has lead monomial $y_2^{q^2}x_2^q$. 
This suggests that a Khovanskii bases for $S_3^{P_3}$ should include $\psi_{[3,1]}(\xi_1)$ and $\psi_{[3,1]}(\xi_2)$.
We know that $S_2^{P_2}$ is generated by $\xi_1$, $\xi_2$ and the orbit products of the variables (Theorem~\ref{orHcal}).
We also know that this set is a Khovanskii basis (Lemma~\ref{orKbasis}). 
Recall that $\overline{S} = \field_q[Y_2,Y_1,X_1,X_2]$.
Using the isomorphism from $S_2^{P_2}$ to $\overline{S}^{P_2}$, we conclude that  
$$\{N(y_2), N(y_1), N(x_1), N(x_2), \psi_{[3,1]}(\xi_1),\psi_{[3,1]}(\xi_2)\}$$ is a Khovanskii basis for the subalgebra $\overline{S}^{P_2}\subset S_3^{P_3}$.
Therefore the \tat s $\psi_{[3,1]}(\xi_1)^q-N(x_2)^{q-1}\psi_{[3,1]}(\xi_2)$
and $\psi_{[3,1]}(\xi_2)^q-N(x_2)^qN(y_2)$ subduct to zero in $\overline{S}^{P_2}$
(these subduction are also subductions in $S_3^{P_3}$) .
Since $\psi_{[3,1]}(\xi_3)-N(x_2)N(y_2)^q$ subducts to zero in $\overline{S}^{P_2}$,
$\psi_{[3,1]}(\xi_3)$ is not needed in the Khovanskii basis and
$\xi_3^q-x_3^{q-1}\xi_4$ subducts to zero in $S_3^{P_3}$. 
To prove that adjoining $\psi_{[3,1]}(\xi_1)$ and $\psi_{[3,1]}(\xi_2)$ to our generating set gives a Khovanskii basis, we need only show that the \tat\ 
$\xi_4^q-x_3^qN(y_3)$ subducts to zero. We know that this \tat\ subducts to zero over the Hook invariants but it is not immediately clear that this subduction can be constructed over the Sylow invariants. It follows from the proof of Theorem \ref{sylow_thm} that $S_3^{P_3}$ is the free module over the algebra generated by $\HH_0$ with module generators
$$\BH\mathcal{C}=\{ \xi_1^{a_1}\xi_2^{a_1}\xi_3^{a_3}\xi_4^{a_4}\psi_{[3,1]}(\xi_1)^{b_1}\psi_{[3,1]}(\xi_2)^{b_2}\mid a_i<q,\, b_i<q\}.$$
Therefore the Hilbert series of $S_3^{P_3}$, and its lead term algebra, is given by 
$$HS(S_3^{P_3},t)=HS(\HH_0,t)  \sum_{\gamma \in \BH\mathcal{C}} t^{\deg(\gamma)}$$
where 
$$HS(\HH_0,t)=HS(\lt(\HH_0),t)=\prod_{i=1}^3\frac{1}{(1-t^{q^{i-1}})(1-t^{q^{5-i}})}\ . $$
The lead monomials of $\BH\mathcal{C}$ are given by
$$\LM(\BH\mathcal{C})=\{ x_3^{a_1+a_2+a_3+a_4}y_3^{a_1+a_2q+a_3q^2+a_4q^3}x_2^{q(b_1+b_2)}y_2^{qb_1+q^2b_2}\mid a_i<q,\, b_i<q\}.$$
Therefore $\LM(\BH\mathcal{C})$ generates a free module of rank $q^6$ over $\lt(\HH_0)$.
The Hilbert series of the vector space span of $\LM(\BH\mathcal{C})$ is 
\begin{align*}
HS(\lt(\BH\mathcal{C},t))&=\Big(\sum_{b_1=0}^{q-1} t^{2b_1 q}\Big) \Big(\sum_{b_2=0}^{q-1} t^{b_2 (q+q^2)}\Big)
\prod_{i=1}^4\Bigg(\sum_{a_i=0}^{q-1} t^{a_i(q^{i-1}+1)}\Bigg)\\
&= \Big(\frac{1-t^{2q^2}}{1- t^{2q}}\Big) \Big(\frac{1-t^{q^2+q^3}}{1-t^{q+q^2}}\Big) 
\prod_{i=1}^4\Bigg(\frac{1-t^{q^i+q}}{1-t^{q^{i-1}+1}}\Bigg)\\
&= \sum_{\gamma \in \BH\mathcal{C}} t^{\deg(\gamma)}\ .
\end{align*}
The module generated by $\LM(\BH\mathcal{C})$ over $\lt(\HH_0)$ is a submodule of 
$\lt(S_3^{P_3})$.  By the above this module has the same Hilbert series as 
$\lt(S_3^{P_3})$, and thus they are equal.  
Therefore $\HH_0 \cup \BH\mathcal{C}$ is a Khovanskii basis for $S_3^{P_3}$. 
Define $$\widetilde{\BH}:=\{\xi_1,\xi_2,\xi_3,\xi_4,\psi_{[3,1]}(\xi_1), \psi_{[3,1]}(\xi_2)\}\ .$$
Since the algebra generated by $\LM(\HH_0\cup\widetilde{\BH})$ coincides with the algebra generated by
$\LM(\HH_0\cup\BH\mathcal{C})$, we see that $\HH_0\cup\widetilde{\BH}$ is also a Khovanskii basis for $S_3^{P_3}$.
As a consequence, $S_3^{P_3}$ is a complete intersection with relations constructed by subducting the
\tat s $\xi_1^{q}-x_3^{q-1}\xi_2$,  $\xi_2^{q}-x_3^{q-1}\xi_3$, $\xi_3^q-x_3^{q-1}\xi_4$, $\xi_4^q-x_3^qN(y_1)$, $\psi_{[3,1]}(\xi_1)^q-N(x_2)^{q-1}\psi_{[3,1]}(\xi_2)$
and $\psi_{[3,1]}(\xi_2)^q-N(x_2)^qN(y_2)$.
The two elements $\psi_{[3,1]}(\xi_{1})$ and $\psi_{[3,1]}(\xi_2)$ are
not required in the generating set $\HH_0 \cup \widetilde{\BH}$.
The ring $S_3^{P_3}$ is also a complete intersection using the smaller generating set
formed by omitting $\psi_{[3,1]}(\xi_{1})$ and $\psi_{[3,1]}(\xi_2)$.
The relations associated to the first two \tat s allow us to express 
$\psi_{[3,1]}(\xi_{1})$ and $\psi_{[3,1]}(\xi_2)$ as elements of  
$\field_q[x_3,\xi_1,\xi_2,\xi_3]$:
$$\psi_{[3,1]}(\xi_{1})=\xi_1^q-\xi_2x_3^{q-1}+\xi_1x_3^{2q-2}$$
and
$$\psi_{[3,1]}(\xi_{2})=\xi_2^q-\xi_3 x_3^{q-1}-\xi_1^q x_3^{q^2-q}+\xi_1 x_3^{q^2+q-2}
-\psi_{[3,1]}(\xi_{1}) x_3^{q^2-q}\ .$$
Substituting these
expressions for $\psi_{[3,1]}(\xi_1)$ and $\psi_{[3,1]}(\xi_2)$ into the subductions of the
\tat s $\xi_3^q-x_3^{q-1}\xi_4$ and $\xi_4^q-x_3^qN(y_1)$
gives relations which rewrite $\xi_3^q$ and $\xi_4^q$.
Substituting the expressions for $\psi_{[3,1]}(\xi_{1})$ and $\psi_{[3,1]}(\xi_2)$ 
into the relations associated to the last two \tat s gives relations which rewrite $\xi_1^{q^2}$ and $\xi_2^{q^2}$
Therefore, using the smaller generating set, we can replace the module generators $\BH\mathcal{C}$ with $\{\xi_1^{e_1}\xi_2^{e_2}\xi_3^{e_3}\xi_4^{e_4}\mid e_1,e_2<q^2, \, e_3,e_4<q\}$. 
\end{example}

Let $\HH_0$ denote the system of parameters given by the orbit products of the variables and
define $\BH$ and $\mathcal{C}$ as in the proof of Theorem \ref{sylow_thm}.

\begin{thm} \label{sylowKB_thm} Using the lexicographic order, $\HH_0\cup\BH\mathcal{C}$ is a Khovanskii basis for $S_m^{P_m}$.
Furthermore $S_m^{P_m}$ is a complete intersection.
\end{thm}

\begin{proof} Applying Lemma \ref{monomorphism_lem} inductively shows that the elements of $\BH\mathcal{C}$ are of the form $\prod \psi_{[m,j]}(\xi_{i})^{e_{ij}}$
with $e_{ij}<q$ (recall that $\psi_{[m,0]}$ is the identity map).
Using Lemma~\ref{lt of elements}, careful book-keeping shows that $\LM(\BH\mathcal{C})$
has the same cardinality as $\BH\mathcal{C}$ and that $\LM(\BH\mathcal{C})$ generates a free module over the algebra generated by $\LM(\HH_0)$.
We denote this module by $K$.
It follows from the proof of Theorem~\ref{sylow_thm} that $S_m^{P_m}$ is the free module over
$\HH_0$ generated by $\BH\mathcal{C}$.
Therefore $K$ and $S_m^{P_m}$ have the same Hilbert series.
Since $S_m^{P_m}$ and $\lt(S_m^{P_m})$ have the same Hilbert series,
we have $HS(K,t)=HS(\lt(S_m^{P_m},t))$.
The module $K$ is a subset of $\lt(S_m^{P_m})$. 
Hence $K=\lt(S_m^{P_m})$ and $\HH_0 \cup \BH\mathcal{C}$ is a Khovanskii basis for $S_m^{P_m}$. 
Let $\widetilde{\BH}$ denote the subset of $\BH\mathcal{C}$ 
consisting of elements of the form $\psi_{[m,j]}(\xi_{i})$. Since the algebra generated by $\LM(\HH_0\cup\widetilde{\BH})$ coincides with the algebra generated by
$\LM(\HH_0\cup\BH\mathcal{C})$, we see that $\HH_0\cup\widetilde{\BH}$ is also a Khovanskii basis for $S_m^{P_m}$.

Note that $\widetilde\BH=\{\psi_{[m,j]}(\xi_{i})\mid 0\leq j\leq m-2 \text{ and } 1\leq i\leq n-2-2j\}$.
The non-trivial \tat s are 
$\psi_{[m,j]}(\xi_{i})^q-\psi_{[m,j]}(\xi_{i+1})N(x_{j+1})^{q-1}$ for $i<n-2-2j$
and $\psi_{[m,j]}(\xi_{i})^q-N(y_{j+1})N(x_{j+1})^q$ for $i=n-2-2j$.
Since  $\HH_0\cup\widetilde{\BH}$ is a Khovanskii basis, each of these \tat s subducts to zero.
The resulting relations are independent and generate the ideal of relations.
The number of generators minus the number of independent relations is $n$, proving that $S_m^{P_m}$ is a complete intersection.
\end{proof}

\begin{remark} It follows from the proof of Theorem \ref{sylowKB_thm} that
$\BH\mathcal{C}$ consists of the monomial factors of
$\prod\{\psi_{[m,j]}(\xi_{i})^{q-1}\mid  0\leq j\leq m-2 \text{ and } 1\leq i\leq n-2-2j\}$.
\end{remark}

\begin{cor}\label{Sylow_basis}
    The monomial factors of 
$$\Gamma_0=\prod_{i=1}^{n-2} \xi_{i}^{e_i} \text{ where }
e_i = q^{m-\lceil i/2\rceil}-1\ .$$ 
 form a basis for 
    $S_m^{P_m}$ as a free module over $\field_q[\HH_0]$.
\end{cor}
\begin{proof}
  Using Theorem~\ref{sylowKB_thm} and its proof,
  each of the \tat s  $\xi_{n-2}^q - x_m^q N(y_m)$, 
  $\xi_{n-3}^q - x_m^{q-1}\xi_{n-2}$, 
  $\xi_{n-4}^{q^2} - x_m^{q(q-1)}\xi_{n-3}^q$,
  $\xi_{n-5}^{q^2} - x_m^{q(q-1)}\xi_{n-4}^q, \dots,
  \xi_2^{q^{m-1}}-x_m^{q^{m-2}(q-1)}\xi_3^{q^{m-2}}$, and
  $\xi_1^{q^{m-1}}-x_m^{q^{m-2}(q-1)}\xi_2^{q^{m-2}}$
  subducts to zero.   Applying Theorem~\ref{sylow_thm}
  and comparing the number of monomial factors of
  $\Gamma_0$ with the number of module generators for $S_m^{P_m}$
  over $\field_q[\HH_0]$ yields the result.
\end{proof}

\begin{remark}
    The rank of $S_m^{P_m}$ as a free $\field_q[\HH_0]$-module is $q^{m(m-1)}$ which surprisingly
is also $|P_m|$.
\end{remark}

\begin{cor}\label{Sylow minimality}
    $S_m^{P_m}$ is minimally generated by 
    $\HH_0 \cup \{\xi_1,\xi_2,\dots,\xi_{n-2}\}$.
\end{cor}

\begin{proof}
  In \cite{Ferreira+Fleischmann-fields:16} it is shown that 
$$\F(S_m^{P_m}) = \field_q(N(y_1),N(x_1),N(x_2),\dots,N(x_m),\xi_1,\xi_2,\dots,\xi_{m-1})$$
and hence this set of invariants is algebraically independent.
Clearly 
$$\V(N(y_1),N(x_1),N(x_2),\dots,N(x_m),\xi_1,\xi_2,\dots,\xi_{m-1}) = 
\Span_{\overline{\field}_q}\{e_2,e_2,\dots,e_{m}\}.$$
The remaining $\xi_{i}$ all vanish on this set and so to cut this variety down to a point
(the origin) we need to add in the remaining $m-1$ norms.  It remains to show that
$\xi_{m},\xi_{m+1},\dots,\xi_{n-2}$ are also necessary as generators.

  Otherwise, one of these $\xi_{i}$ can be expressed as a polynomial $f$ in the other $2n-3$
generators.  Using the lex order, since $\lt(\xi_{i})= y_m^{q^{i-1}}x_m$, this must also be the 
lead term of $f$ as 
an element of $S_m$.  This lead term cannot be expressed as a product of the lead terms
of the other $2m-3$ generators and thus it must arise from a \tat.  This \tat\ must
be in degree $q^{i-1}+1$ and involve at least one generator whose lead term is greater than $y_m^{q^{i-1}}x_m$.  But any of these
$2m-3$ generators whose lead term is greater than $y_m^{q^{i-1}}x_m$ also has degree
greater than $q^{i-1}+1$.  Hence no such \tat\ can exist.   
\end{proof}

\subsection{Computing the Borel Invariants} \label{Borel_inv}

Let $B$ denote the Borel subgroup of upper triangular matrices in $\orthp{2m}{q}$. 
The Sylow subgroup $P$ is a normal subgroup of $B$ and we identify $B/P$ with the subgroup of diagonal matrices, which we denote by $T$.
Let $N$ denote the normaliser of $T$ in $\orthp{2m}{q}$.
The Weyl group, $W:=N/T$, is isomorphic to the wreath product $C_2\wr \Sigma_m$. We can identify $W$ with a subgroup of
$\orthp{2m}{q}$: the group $\Sigma_m$ acts by permuting the hyperbolic pairs
and the additional generator, say $\omega$, interchanges $y_1$ and $x_1$
while fixing the other variables.
Note that $B$ and $N$ are not a $BN$-pair for $\orthp{2m}{q}$.
Most of the axioms are satisfied and $\orthp{2m}{q}=BNB$
but $\omega B\omega=B$ (see \cite[p. 85]{taylor-classicalgroups:92}
and \cite[\S1.7.5]{geck-alggrps:03}).
The subgroup $\sorthp{2m}{q}$ contains $B$ and has a $BN$-pair
with Weyl $W'$, an index $2$ subgroup of $W$. The Dynkin type of $W'$ is
$D_m$ (see \cite[\S1.7.9, \S1.3.16]{geck-alggrps:03}).

Let $\HH_0$ denote the homogeneous system of parameters for $S^P$ given by the orbit products over $P$ of the variables 
and let $\field_q[\HH_0]$ denote the algebra generated by $\HH_0$. Define $\BH$ and $\mathcal{C}$ as in the proof of Theorem \ref{sylow_thm}.
Then
$$S^P=\bigoplus_{\gamma\in\BH\mathcal{C}} \field_q[\HH_0]\gamma.$$

\begin{lem} \label{borel_lem} Suppose $w\in S$ with $w$ homogeneous of degree one  and $b\in B$ with $w\cdot b=c w$ for some scalar $c$.
If the $P$-orbit product $N(w)$ has degree $q^\ell$ then $N(w)\cdot b=c^{q^\ell}N(w)$.
\end{lem}
\begin{proof} Since $P$ is normal in $B$, for every $g\in P$, we have $\overline{g}:=b^{-1}gb\in P$.
Thus 
\begin{align*}
N(w)b&=\left(\prod\{wg\mid g \in P\}\right)b=\prod\{wgb\mid g\in P\}=\prod\{wb\overline{g}\mid g\in P\}\\
             &=\prod\{cw\overline{g}\mid g\in P\}=c^{q^\ell}\prod\{wg\mid g\in P\}=c^{q^\ell}N(w)
\end{align*}   
as required.          
\end{proof}

Since the elements of $\HH_0$ are orbit products over $P$, it follows from Lemma \ref{borel_lem} that $B/P$ acts on $\field_q[\HH_0]$.
Arguing as in the proof of Theorem \ref{sylowKB_thm}, the elements of $\BH\mathcal{C}$ are monomials in  $\psi_{[m,j]}(\xi_{i})$. 
Since $\psi_{[m,j]}(\xi_{i})$ is Borel invariant, we have 
$$S^B=(S^P)^{B/P}=\bigoplus_{\gamma\in\BH\mathcal{C}} \field_q[\HH_0]^{B/P}\gamma.$$
Therefore $S^B$ is a free module over $\field_q[\HH_0]^{B/P}$ and to compute $S^B$, we need only compute $\field_q[\HH_0]^{B/P}$.

\begin{thm} \label{Bor_inv} $S^B$ is a complete intersection. 
Furthermore $\widetilde{\HH_0}\cup \BH\mathcal{C}$ is a Khovanskii basis for $S^B$, using the lexicographic order, where 
$$\widetilde{\HH_0}:=\{N(y_i)^{q-1},N(x_i)^{q-1}, N(y_i)N(x_i)\mid i=1,\ldots m\} .$$
\end{thm}
\begin{proof} The proof follows from the computation of $\field_q[\HH_0]^{B/P}$.
From Lemma \ref{borel_lem},
$N(y_i)^{q-1}$ and $N(x_i)^{q-1}$ are Borel invariant. 
Furthermore, $B$ contains the linear transformation which scales $y_i$ by $c\in\field_q$, takes $x_i$ to $c^{-1}x_i$, and fixes the other variables.
Therefore no smaller power of $N(y_i)$ or $N(x_i)$ is Borel invariant. Any element of $B/P$ which scales $y_i$ by $c$ must scale $x_i$ by $c^{-1}$.
Therefore $N(y_i)N(x_i)$ is Borel invariant. It follows from this that  $\field_q[\HH_0]^{B/P}$ is the complete intersection generated by
$\widetilde{\HH_0}$ subject to the relations 
$(N(y_i)N(x_i))^{q-1}=N(y_i)^{q-1}N(x_i)^{q-1}$ and $\widetilde{\HH_0}$ and is a Khovanskii basis for $\field_q[\HH_0]^{B/P}$. 
Let $\widetilde{K}$ denote the algebra generated by the lead monomials of $\widetilde{\HH_0}$.
Since $\widetilde{\HH_0}$ is a Khovanskii basis for $\field_q[\HH_0]^{B/P}$,
the Hilbert series of $\field_q[\HH_0]^{B/P}$ equals the Hilbert series of $\widetilde{K}$.
Furthermore
$$HS(S^B,t)=HS(\field_q[\HH_0]^{B/P},t)
\left(\sum_{\gamma\in\BH\mathcal{C}} t^{\deg(\gamma)} \right)\ .$$
Therefore
$$HS(\lt(S^B),t)=HS(S^B,t)=HS(\widetilde{K},t)\left(\sum_{\gamma\in\BH\mathcal{C}} t^{\deg(\gamma)} \right)$$
 which proves that  $\widetilde{\HH_0}\cup \BH\mathcal{C}$ is a Khovanskii basis for $S^B$.
This means that the non-trivial \tat s from $\BH\mathcal{C}$ subduct to zero in $S^B$.
These subductions, along with the relations coming from $\widetilde{\HH_0}$, give a generating set for the ideal of relations, proving that $S^B$ is a complete intersection.
\end{proof}

Let $\R$ denote the Reynolds operator from $S^B$ to $S^G$.
Since the order of $B$ is $q^{m(m-1)}(q-1)^m$, the index
$$[G:B]=
2\frac{q^m-1}{q-1}\prod_{j=1}^{m-1}\frac{q^{2j}-1}{q-1}
\equiv_{(p)} 2$$
and $$\R(f)=\frac{1}{2}\tr_B^G(f)
=\frac{1}{2}\sum_{Bg\in B\backslash G} fBg$$
To compute the Reynolds operator we need to choose left coset representatives for $B$ in $\orthp{2m}{q}$.
We use the double coset partition to do this. 
The double cosets, $BwB$ for $w\in W$, partition $\orthp{2m}{q}$ so we can choose coset representatives of the form $wb$ for $w\in W$ and $b\in B$. Observe that the number of left cosets in the double coset
$BwB$ is given by the index 
$[B:B\cap w^{-1}Bw]=[P:P\cap w^{-1}Pw]$,
which divides $q^{m(m-1)}$.

\begin{thm} \label{ren_thm} Using the lexicographic order on $S_m$, 
$$\lt\left(\R(N(y_m)^{q-1})\right)=y_m^{(q-1)q^{n-2}}.$$
\end{thm} 

\begin{proof}
Note that $y_m^{(q-1)q^{n-2}}$ is the largest monomial of 
degree $(q-1)q^{n-2}$.
The elements of $W$ permute the variables. Every term of $N(y_m)$ is divisible by $y_m$. Therefore the lead term of $N(y_m)^{q-1}w$ is
$y_m^{(q-1)q^{n-2}}$ if and only if $w$ stablises $y_m$. Since elements of $B$ do not increase the $y_m$-degree, the only coset representatives that contribute to the coefficient of
$y_m^{(q-1)q^{n-2}}$ in $\R(N(y_m)^{q-1})$ are those of the form $wb$ were $w$ stablises $y_m$.
If $w$ stabilises $y_m$, then all of the coset representatives of the form $wb$ make the same contribution to the coefficient. The number of coset representatives associated to a double coset
is either $1$, if $P\cap w^{-1}Pw=P$, or congruent  to $0$ modulo $p$.
The only elements of $W$ for which $P\cap w^{-1}Pw=P$ are $1$ and $\omega$.
Therefore the coefficient of $y_m^{(q-1)q^{n-2}}$ in $\R(N(y_m)^{q-1})$ is $(1/2)2=1$.
\end{proof}

\section{The Invariants of \texorpdfstring{$\orthp{4}{q}$}{}}\label{base_case}

The ring of invariants for $\orthp{4}{q}$ was computed by
Huah Chu \cite{chu-polyinvorth:01}. 
In this section we provide an alternative computation and show that
Lemma~\ref{oplus_tech-lem} is satisfied for $m=2$.  This provides the base case for the inductive proof of the lemma and implies that Theorem~\ref{oplus_thm_v2} holds for $m=2$.

The group $\orthp{2}{q}$ has order $2(q-1)$ and the ring of invariants is generated by $\xi_1=y_1x_1=u_1$ and $d_{1,1}=y_1^{q-1}+x_1^{q-1}=\xi_2/\xi_1$.

\begin{lem}\label{u2lem} $u_2=\xi_1^q(\xi_3+c_{3,2})-\xi_2^{q+1}$ with $c_{3,2}\in R_2$. 
 Furthermore 
 $$c_{3,2}=\xi_1\xi_2^{q-1}+\xi_1^{q+2}\xi_2^{q-3}-\sum_{j=2}^{(q-1)/2}(-1)^j\xi_2^{q-1-2j}\xi_1^{1+j(q+1)}\frac{(q-2-j)!}{j!((q-1)-2j)!}$$
 and $\ker(\Phi_{3,1})=u_2R_3$.
 \end{lem}
 \begin{proof} Working in $S_1$, we have $\xi_2=\xi_1d_{1,1}$ where $d_{1,1}=y^{q-1}+x^{q-1}$. Applying $\PP^q$, using the Cartan identity and the stability condition, gives
 $\xi_3=\xi_2\PP^{q-1}(d_{1,1})+\xi_1^q\PP^{q-2}(d_{1,1})=\xi_2d_{1,1}^q+\xi_1^q\PP^{q-2}(d_{1,1})$. 
 Multiplying by $\xi_1^q$ gives $\xi_1^q\xi_3=\xi_2^{q+1}+\xi_1^{2q}\PP^{q-2}(d_{1,1})$.
 We will show that $\xi_1^{q}\PP^{q-2}(d_{1,1})\in R_2$. From this it, follows that the relation lifts to $S_2$, giving a non-zero element of $\ker(\Phi_{3,1})$.
 Using Lemma~\ref{phi_bar_lemma}, the element is divisible by $u_2$
 Comparing lead terms using the lexicographic order in $S_2$ gives $u_2=\xi_1^q(\xi_3+c_{3,2})-\xi_2^{q+1}$ with 
 $c_{3,2}=-\xi_1^{q}\PP^{q-2}(d_{1,1})$.
 A simple calculation gives $\PP^{q-2}(d_{1,1})=-(y^{q^2-2q+1}+x^{q^2-2q+1})$. Taking $a=y^{q-1}$ and $b=x^{q-1}$, we have $-\PP^{q-2}(d_{1,1})=a^{q-1}+b^{q-1}$.
 We can write the symmetric polynomial $a^{q-1}+b^{q-1}$ as a polynomial in $\sigma_1=a+b$ and $\sigma_2=ab$. A closed form for this expression over the integers is
 $$a^{q-1}+b^{q-1}=(q-1)\sum_{j=0}^{(q-1)/2}(-1)^j\sigma_1^{q-1-2j}\sigma_2^j\frac{(q-2-j)!}{j!((q-1)-2j)!},$$
 see \cite{gould:99} formula (1) or (10) (note there is a typographical error in \cite{gould:99}(10)
 where $i+j+1$ is written but $i+j-1$ is intended).   
 Observe that $\sigma_2=\xi_1^{q-1}$, $\sigma_1=d_{1,1}$ and $\xi_1\sigma_1=\xi_2$. Multiplying the above expression by $\xi_1^q$ and substituting
 gives
 $$\xi_1^q\PP^{q-2}(d_{1,1})=\sum_{j=0}^{(q-1)/2}(-1)^j\xi_2^{q-1-2j}\xi_1^{1+j(q+1)}\frac{(q-2-j)!}{j!((q-1)-2j)!}\in R_2,$$
 as required.
\end{proof}

\begin{lem} 
    $$(-1)^j (q-1)\frac{(q-2-j)!}{j!((q-1)-2j)!} \equiv_{(p)} \Cat(j)$$
    where $\Cat(j) = \frac{1}{j+1}\binom{2j}{j}$ is the $j^\text{th}$ Catalan number.
In particular $$c_{3,2} = \sum_{j=0}^{(q-1)/2} \Cat(j)\, \xi_1^{j(q+1)+1} \xi_2^{q-1-2j}.$$
\end{lem}

\begin{proof}
Both of the expressions $(-1)^j(q-1)\frac{(q-2-j)!}{j!((q-1)-2j)!}$ and 
$\frac{1}{j+1}\binom{2j}{j}$ are integers.  We show they are equal when 
viewed as elements of $\field_p$:
\begin{align*}
        (-1)^j (q-1)\frac{(q-2-j)!}{j!((q-1)-2j)!} & =
         (-1)^{j+1}\frac{(q-2-j)!}{j!((q-1)-2j)!}\\
         & =\frac{(-1)^{j+1}}{j!} (q-2j)(q-2j+1)\cdots(q-2-j)\\
         & = \frac{(-1)^{j+1}}{j!} (-2j)(-2j+1)\cdots(-2-j)\\
        & = \frac{(-1)^{j+1}}{j!}(-1)^{j-1} (2j)(2j-1)\cdots(j+2)\\
        & = \frac{(2j)!}{j!(j+1)!} 
         = \frac{1}{j+1}\binom{2j}{j}\ . \qedhere
    \end{align*}

\end{proof}

\begin{lem}\label{u2_div_lem}
If $u_2$ divides $f$ in $S_2$ and $f\in R_3$ then $f/u_2\in R_3$.
\end{lem}
\begin{proof} From Lemma~\ref{u2lem}, we have $\ker(\Phi_{3,1})=u_2R_3$.
The result follows using Lemma~\ref{phi_bar_lemma}(b) and the argument given in
Remark~\ref{um_div_rem}.
\end{proof}

\begin{lem} \label{u2_st} (a)  $\PP^i(u_2)=0$ for $0<i<q$.\\
 (b) $\PP^i(u_2)/u_2\in R_3$ for $i\not\in\{q^2,q^2+q,q^2+2q\}$.\\
 \end{lem}
 \begin{proof} As an element of $S_2$, $u_2$ is a product of pairwise relatively prime linear forms. 
 Therefore, using Lemma~\ref{steenrod on linear forms}, $u_2$ divides $\PP^i(u_2)$ in $S_2$ for all $i$.
 If $\PP^i(u_2)\in R_3$, it follows from Lemma~\ref{u2_div_lem} that
 $\PP^i(u_2)/u_2\in R_3$.
 Recall from Lemma~\ref{u2lem} that $u_2=\xi_1^q(\xi_3+c_{3,2})-\xi_2^{q+1}$ with $c_{3,2}\in R_2$.
 Therefore, using the Cartan identity, Lemma~\ref{steenrod on q powers},
 and Corollary~\ref{com_st}, we have $\PP^i(u_2)\in R_2\subset R_3$
 for $0<i<q$. 
 Since $u_2$ is linear in $\xi_3$, this give $\PP^i(u_2)=0$ for $0<i<q$, proving (a).
 Again using Corollary~\ref{com_st}, for $i\not\in\{q^2,q^2+q,q^2+2q\}$, we have $\PP^i(u_2)\in R_3$. 
 Therefore, if $i\not\in\{q^2,q^2+q,q^2+2q\}$, we have
 $\PP^i(u_2)/u_2\in R_3$, proving (b).
 \end{proof}

\begin{lem} \label{c22_st} (a) $\PP^1(c_{3,2})=\xi_2^q$.\\
 (b) $\PP^i(c_{3,2})=0$ for $1<i<q$.\\
 (c) $\PP^{q^2}(c_{3,2})\equiv_{\langle\xi_1^{q^2}\rangle}\xi_2^{q^2-q+1}-\xi_1^q\xi_2^{q^2-2q}\xi_3$.\\
 (d) $\PP^{q^2-q}(c_{3,2})\equiv_{\langle\xi_1^{q^2}\rangle}\xi_1\xi_3^{q-1}-\xi_1^q\xi_2^{q^2-2q+1}+\xi_1^{2q}\xi_2^{q^2-3q}\xi_3 $.
 \end{lem}
 \begin{proof} Using Lemma~\ref{u2_st}, we have $\PP^1(u_2)=0$. Thus
 $$0=\PP^1(\xi_1^q(\xi_3+c_{3,2})-\xi_2^{q+1})=\xi_1^q(\xi_2^q+\PP^1(c_{3,2}))-2\xi_1^q\xi_2^q.$$
 Therefore $\PP^1(c_{3,2})=\xi_2^q$, proving (a).
 Again using Lemma~\ref{u2_st}, for $1<i<q$, we have 
 $$0=\PP^i(u_2)=\PP^i(\xi_1^q(\xi_3+c_{3,2})-\xi_2^{q+1})=\xi_1^q\PP^i(c_{3,2}),$$
 giving $\PP^i(c_{3,2})=0$.
  Since $c_{3,2}\equiv_{\langle\xi_1^q\rangle}\xi_1\xi_2^{q-1}$, it follows from Lemma \ref{sub_max} that
  $$\PP^{q^2}(c_{3,2})\equiv_{\langle\xi_1^{q^2}\rangle}\PP^{q^2}(\xi_1\xi_2^{q-1})=\xi_2^{q^2-q+1}-\xi_1^q\xi_2^{q^2-2q}\xi_3,$$
  proving (c).  To prove (d), we first compute 
  $$\PP^{q^2-q}(\xi_1\xi_2^{q-1})=\xi_1\PP^{q^2-q}(\xi_2^{q-1})+\xi_2\PP^{q^2-q-1}(\xi_2^{q-1})+\xi_1^q\PP^{q^2-q-2}(\xi_2^{q-1}).$$
Note that $\PP^{q^2-q}(\xi_2^{q-1})=\xi_3^{q-1}$, $\PP^{q^2-q-1}(\xi_2^{q-1})=-2\xi_1^q\xi_2^{q^2-2q}$ and
$$\PP^{q^2-q-2}(\xi_2^{q-1})=-\xi_2^{q^2-2q+1}-\xi_3\PP^{q^2-2q-2}(\xi_2^{q-2})= -\xi_2^{q^2-2q+1}+4\xi_3\xi_1^q\xi_2^{q^2-3q}.$$
Therefore $\PP^{q^2-q}(\xi_1\xi_2^{q-1})=\xi_1\xi_3^{q-1}-3\xi_1^q\xi_2^{q^2-2q+1}+4\xi_1^{2q}\xi_2^{q^2-3q}\xi_3$.
Observe that $\PP^{q^2-q}(\xi_1^{q+2}\xi_2^{q-3})\equiv_{\langle\xi_1^{q^2}\rangle} \xi_2^q\PP^{q^2-2q}(\xi_1^2\xi_2^{q-3})$
and $$\PP^{q^2-2q}(\xi_1^2\xi_2^{q-3})=2\xi_1^q\xi_2^{q^2-3q+1}-3\xi_1^{2q}\xi_2^{q^2-4q}\xi_3.$$
Therefore $\PP^{q^2-q}(\xi_1^{q+2}\xi_2^{q-3})\equiv_{\langle\xi_1^{q^2}\rangle}2\xi_1^q\xi_2^{q^2-2q+1}-3\xi_1^{2q}\xi_2^{q^2-3q}\xi_3$.
Suppose $b$ is homogeneous of degree $q^2-4q+1$. Then 
$\PP^{q^2-q}(\xi_1^{2q}b)\equiv_{\langle\xi_1^{q^2}\rangle} \xi_2^q\PP^{q^2-2q}(\xi_1^qb)$ and, using Lemma \ref{sub_max},
$\PP^{q^2-2q}(\xi_1^qb) =\xi_1^{q^2}\PP^{q^2-4q}(b)$. Therefore $\PP^{q^2-q}(\xi_1^{2q}b)\equiv_{\langle\xi_1^{q^2}\rangle} 0$
and 
\begin{eqnarray*}
\PP^{q^2-q}(c_{3,2})&\equiv_{\langle\xi_1^{q^2}\rangle}&\PP^{q^2-q}(\xi_1\xi_2^{q-1})+\PP^{q^2-q}(\xi_1^{q+2}\xi_2^{q-3})\cr
&\equiv_{\langle\xi_1^{q^2}\rangle}& \xi_1\xi_3^{q-1}-\xi_1^q\xi_2^{q^2-2q+1}+\xi_1^{2q}\xi_2^{q^2-3q}\xi_3
\end{eqnarray*}
as required. 
 \end{proof}

By Definition, we have $\PP^{q^2}(u_2)=u_2d_{1,2}$ and $\PP^{q^2+q}(u_2)=u_2d_{2,2}$.

\begin{lem}\label{u2dlem} (a) $u_2d_{1,2}-\xi_1^q\xi_4\in R_3$ and
 $$u_2d_{1,2}\equiv_{\langle \xi_1^{q^2}\rangle} \xi_1^q\xi_4-\xi_2\xi_3^q+\xi_1\xi_2^q\xi_3^{q-1}.$$
 (b) $u_2d_{2,2}-\xi_2^q\xi_4\in R_3$ and
 $$u_2d_{2,2}\equiv_{\langle \xi_1^{q^2}\rangle} \xi_2^q\xi_4-\xi_3^{q+1}-\xi_1^q\xi_2^{q^2-q}\xi_3.$$
 \end{lem}
 
 \begin{proof} Using Lemma~\ref{u2lem}, Corollary~\ref{com_st} and the Cartan identity, 
 $$\PP^{q^2}(u_2)=\xi_1^q(\xi_4 +\PP^{q^2}(c_{3,2}))-\xi_2\xi_3^q+\xi_2^q\PP^{q^2-q}(c_{3,2})+\xi_1^{q^2}\PP^{q^2-2q}(c_{3,2}).$$
 Since $c_{3,2}\in R_2$, we have $\PP^i(c_{3,2})\in R_3$ and $u_2d_{1,2}-\xi_1^q\xi_4\in R_3$.
 Using Lemma \ref{c22_st}, we have 
 \begin{eqnarray*}u_2d_{1,2}=\PP^{q^2}(u_2)&\equiv_{\langle \xi_1^{q^2}\rangle} &\xi_1^q(\xi_4 +\xi_2^{q^2-q+1}-\xi_1^q\xi_2^{q^2-2q}\xi_3) -\xi_2\xi_3^q\cr
 &&+\xi_2^q(\xi_1\xi_3^{q-1}-\xi_1^q\xi_2^{q^2-2q+1}+\xi_1^{2q}\xi_2^{q^2-3q}\xi_3 )\cr
 &\equiv_{\langle \xi_1^{q^2}\rangle} & \xi_1^q\xi_4  -\xi_2\xi_3^q+\xi_1\xi_2^q\xi_3^{q-1}
 \end{eqnarray*}
 as required. 
 Using the stability condition
 $$ \PP^{q^2+q}(u_2)=\xi_2^q(\xi_4+\PP^{q^2}(c_{3,2}))-\xi_3^{q+1}-\xi_2^{q^2+1}+\xi_1^{q^2}\PP^{q^2-q}(c_{3,2}).$$
 Using Lemma \ref{c22_st} gives
 $$\PP^{q^2+q}(u_2)\equiv_{\langle \xi_1^{q^2}\rangle} \xi_2^q\xi_4 -\xi_3^{q+1} -\xi_2^{q^2+1}+\xi_2^q(\xi_2^{q^2-q+1}-\xi_1^q\xi_2^{q^2-2q}\xi_3).$$
 Therefore $u_2d_{2,2} \equiv_{\langle \xi_1^{q^2}\rangle} \xi_2^q\xi_4-\xi_3^{q+1}-\xi_1^q\xi_2^{q^2-q}\xi_3$, as required.
 \end{proof}

\begin{lem}\label{u2d_st} (a) For $0<i<q$, $\PP^i(d_{1,2})$ and $\PP^i(d_{2,2})$ lie in $R_3$.\\
 (b) For $0<i<q^3$, write $i=\ell q^2+r$ with $r<q^2$. Then $u_2^{\ell+1}\PP^i(d_{1,2})$ and  $u_2^{\ell+1}\PP^i(d_{2,2})$
 lie in $R_4$ with $\xi_4$-degree at most $\ell+1$.
 \end{lem}
 \begin{proof} Suppose $0<i<q$. Using Lemma \ref{u2_st}, $\PP^i(u_2)=0$. Therefore $\PP^i(u_2d_{1,2})=u_2\PP^i(d_{1,2})$ and
 $\PP^i(u_2d_{2,2})=u_2\PP^i(d_{2,2})$. Since $u_2d_{1,2}-\xi_1^q\xi_4\in R_3$ and $i<q$, we have 
 $\PP^i(u_2d_{1,2}-\xi_1^q\xi_4)\in R_3$. 
 Since $\PP^1(\xi_1^q\xi_4)=\xi_1^q\xi_3^q$ and $\PP^i(\xi_1^q\xi_4)=0$ for $1<i<q$, 
 we see that 
 $u_2\PP^i(d_{1,2})=\PP^i(u_2d_{1,2})\in R_3$.
 It then follows from Lemma~\ref{u2_div_lem} that $\PP^i(d_{1,2})\in R_3$.
 Similarly, since $u_2d_{2,2}-\xi_2^q\xi_4\in R_3$, we have 
 $u_2\PP^i(d_{2,2})=\PP^i(u_2d_{2,2})\in R_3$, proving $\PP^i(d_{2,2})\in R_3$,
 completing the proof of (a).
 
The proof of (b) is by induction on $i$. Part (a) gives (b) for $i<q$. 
Suppose  $i=\ell q^2+r$ with $r<q^2$ and $\ell<q$. Using Lemma \ref{com_st}, for $f\in R_3$ we have $\PP^i(f)\in R_4$ with $\xi_4$-degree at most $\ell$.
 Using Lemma~\ref{u2dlem}(a), $u_2d_{1,2}-\xi_1^q\xi_4\in R_3$. Therefore $\PP^i(u_2d_{1,2})\in R_4$ with $\xi_4$-degree at most $1$ when $\ell=0$ and at most $\ell$ otherwise.
 Similarly, using Lemma~\ref{u2dlem}(b), $u_2d_{2,2}-\xi_2^q\xi_4\in R_3$. Thus  $\PP^i(u_2d_{2,2})\in R_4$ with 
 $\xi_4$-degree at most $1$ when $\ell=0$ and at most $\ell$ otherwise. In the following, let $d$ denote either $d_{1,2}$ or $d_{2,2}$.
 Using the Cartan identity,
 $$u_2\PP^i(d)=\PP^i(u_2d)-\sum_{j>0}\PP^j(u_2)\PP^{i-j}(d).$$
 We have seen that $\PP^i(u_2d)\in R_4$ with $\xi_4$-degree at most $\ell+1$. 
 To prove (b), it is sufficient to show that for $j>0$, we have $u_2^\ell \PP^j(u_2)\PP^{i-j}(d)\in R_4$ with
 $\xi_4$-degree at most $\ell+1$. 
 Write $i-j=\ell' q^2+r'$ with $r'<q^2$ and $\ell'\leq\ell$. By induction $u_2^{\ell'+1}\PP^i(d)\in R_4$
 with $\xi_4$-degree at most $\ell'+1$.
 
 Suppose $j<q^2$. Using Lemma \ref{u2_st}, we have $\PP^j(u_2)/u_2\in R_3$.
 Therefore $$u_2^\ell \PP^j(u_2)\PP^{i-j}(d)=(\PP^j(u_2)/u_2)u_2^{\ell+1}\PP^{i-j}(d)\in R_4$$ with $\xi_4$-degree at most $\ell+1$, as required.
 
For $q^2\leq j< q^3$, we  have $\ell'+1\leq \ell$ and $\PP^j(u_2)\in R_4$ with $\xi_4$-degree at most $1$.
Therefore $u_2^\ell\PP^{i-j}(d)\in R_4$ with $\xi_4$-degree at most $\ell$ and $u_2^\ell \PP^j(u_2)\PP^{i-j}(d)\in R_4$ with
 $\xi_4$-degree at most $\ell+1$, as required.
\end{proof}

\begin{lem}\label{special_u2d_st} 
Suppose $i=(q-1)q^{2}+r$ with $r<q^{2}$ and $i\not\equiv_{(q)}0$. 
Then $u_2^q\PP^i(d_{j,2})$ has $\xi_4$-degree at most $q-1$
\end{lem} 
\begin{proof} From Lemma \ref{u2d_st},
we know that $u_2^q\PP^i(d_{j,2})\in R_4$ with $\xi_4$-degree at most $q$.
We will show that the $\xi_4$-degree at most $q-1$.
Using the Cartan identity
$$u_2^q\PP^i(d_{j,2})=\PP^i(u_2^qd_{j,2})-\sum_{k>0}\PP^k(u_2)^q\PP^{i-qk}(d_{j,2}).$$
Since $u_2^qd_{j,2}\in R_4$ with $\xi_4$-degree $1$, we see that $\PP^i(u_2^qd_{j,2})\in R_4$
with $\xi_4$-degree at most $q$. However, by Lemma \ref{u2dlem}, the coefficient of $\xi_4$ in $u_2^qd_{j,2}$
is $u_2^{q-1}(u_1d_{j-1,1})^q$, which means that the coefficient of 
$\xi_4\xi_3^{q-1}$ in  
 $u_2^qd_{j,2}$ is $(u_{1}^q d_{j-1,1})^q$. 
 Therefore the coefficient of $\xi_{4}^q$ in $\PP^i(u_2^qd_{j,2})$
 is $\PP^r((u_{1}^q d_{j-1,1})^q)$. 
 Since $r\not\equiv_{(q)} 0$, $\PP^r((u_{1}^q d_{j-1,1})^q)=0$
 and the $\xi_{4}$-degree of $\PP^i(u_2^qd_{j,2})$ is at most $q-1$. 
 Since $i<q^3$ and $qk<i$, we have $k<q^{2}$.
 Thus $\PP^k(u_2)/u_2\in R_{3}$ and $$\PP^k(u_2)^q\PP^{i-qk}(d_{j,2})=(\PP^k(u_2)/u_2)^q u_2^q \PP^{i-qk}(d_{j,2}).$$
 Write $r=r_0+qr'$ with $0<r_0<q$. If $r'=0$, then $i-kq<(q-1)q^{2}$ and 
 $u_2^q \PP^{i-qk}(d_{j,2})$ has $\xi_{4}$-degree at most $q-1$.
 If $r'>0$, then the result follows by induction on $r'$.
\end{proof}

\begin{lem} \label{dg02lem} In $S_2^{\orthp{4}{q}}$, $\xi_4=(\xi_3+c_{3,2})d_{1,2}-\xi_2d_{2,2}-c_{4,2}$ with $c_{4,2}\in R_3$ and
 $$-c_{4,2}\equiv_{\langle \xi_1^{q+2}\rangle} \xi_1^2\xi_2^{q-2}\xi_3^{q-1}+\xi_1^q\xi_2^{q^2-2q}\xi_3.$$
 \end{lem}
 \begin{proof} Using Lemma \ref{u2lem}, $u_2=\xi_1^q(\xi_3+c_{3,2})-\xi_2^{q+1}$ with $c_{3,2}\in R_2$.
 Since $u_2d_{1,2}-\xi_1^q\xi_4$ and $u_2d_{2,2}-\xi_2^q\xi_4$ are elements of $R_3$
 (see Lemma~\ref{u2dlem}),
 $$u_2\xi_4-(\xi_3+c_{3,2})u_2d_{1,2}+\xi_2u_2d_{2,2}\in R_3.$$
 Since this expression is divisible by $u_2$ in $S_2$ and lies in $R_3$, it is divisible by $u_2$ in $R_3$ (see Lemma~\ref{u2_div_lem}). 
 Defining $-c_{4,2}:= (u_2\xi_4-(\xi_3+c_{3,2})u_2d_{1,2}+\xi_2u_2d_{2,2})/u_2$ gives 
 $$\xi_4=(\xi_3+c_{3,2})d_{1,2}-\xi_2d_{2,2}-c_{4,2}\ .$$ with $c_{4,2}\in R_3$.
 Using Lemma~\ref{u2dlem}
 \begin{eqnarray*}
 (\xi_3+c_{3,2})u_2d_{1,2}-\xi_2u_2d_{2,2}&\equiv_{\langle \xi_1^{q^2}\rangle}&\xi_4((\xi_3+c_{3,2})\xi_1^q-\xi_2^{q+1})-\xi_3(\xi_2\xi_3^q-\xi_1\xi_2^q\xi_3^{q-1})\cr
 &&+\xi_2(\xi_3^{q+1}+\xi_1^q\xi_2^{q^2-q}\xi_3)-c_{3,2}(\xi_2\xi_3^q-\xi_1\xi_2^q\xi_3^{q-1})\cr
 &=& u_2\xi_4+\xi_1\xi_2^q\xi_3^q+\xi_1^q\xi_2^{q^2-q+1}\xi_3\cr
 &&-c_{3,2}(\xi_2\xi_3^q-\xi_1\xi_2^q\xi_3^{q-1})
 \end{eqnarray*}
 Therefore, using Lemma~\ref{u2lem} to substitute for $c_{3,2}$,
  \begin{eqnarray*}
-u_2c_{4,2}&\equiv_{\langle \xi_1^{q^2}\rangle}&-\xi_1\xi_2^q\xi_3^q-\xi_1^q\xi_2^{q^2-q+1}\xi_3+c_{3,2}(\xi_2\xi_3^q-\xi_1\xi_2^q\xi_3^{q-1})\cr
&\equiv_{\langle \xi_1^{q+2}\rangle}&-\xi_1\xi_2^q\xi_3^q-\xi_1^q\xi_2^{q^2-q+1}\xi_3+\xi_1\xi_2^{q-1}(\xi_2\xi_3^q-\xi_1\xi_2^q\xi_3^{q-1})\cr
&=&-\xi_1^2\xi_2^{2q-1}\xi_3^{q-1}-\xi_1^q\xi_2^{q^2-q+1}\xi_3\cr
&=&-\xi_2^{q+1}(\xi_1^2\xi_2^{q-2}\xi_3^{q-1}+\xi_1^q\xi_2^{q^2-2q}\xi_3).
\end{eqnarray*}
 Thus $-c_{4,2}\equiv_{\langle \xi_1^{q+2}\rangle} \xi_1^2\xi_2^{q-2}\xi_3^{q-1}+\xi_1^q\xi_2^{q^2-2q}\xi_3.$
 \end{proof}
 
Applying $\PP^1$ to $\xi_4=(\xi_3+c_{3,2})d_{1,2}-\xi_2d_{2,2}-c_{4,2}$ and using $\PP^1(c_{3,2})=\xi_2^q$ from Lemma \ref{c22_st} gives
 $$\xi_3^q=2\xi_2^q d_{1,2}-2\xi_1^q d_{2,2}-\PP^1(c_{4,2})+(\xi_3+c_{3,2})\PP^1(d_{1,2})-\xi_2\PP^1(d_{2,2}).$$
 Since $c_{4,2}\in R_3$, $\PP^1(c_{4,2})\in R_3$. Furthermore, using Lemma \ref{dg02lem},
\begin{eqnarray*}
    -\PP^1(c_{4,2})&\equiv_{\langle \xi_1^{q+1}\rangle}&
    \PP^1(\xi_1^2\xi_2^{q-2}\xi_3^{q-1}+\xi_1^q\xi_2^{q^2-2q}\xi_3)\\
    &\equiv_{\langle \xi_1^{q+1}\rangle}& 
    2\xi_1\xi_2^{q-1}\xi_3^{q-1}-\xi_1^2\xi_2^{2q-2}\xi_3^{q-2}+\xi_1^q\xi_2^{q^2-q}\ .
\end{eqnarray*}
 Using Lemmas~\ref{u2_div_lem}, \ref{u2_st}, \ref{u2dlem} and \ref{u2d_st}, $\PP^1(u_2d_{1,2})=u_2\PP^1(d_{1,2})$, $\PP^1(d_{1,2})\in R_3$ and
 $$u_2\PP^1(d_{1,2})\equiv_{\langle \xi_1^{q^2}\rangle}\xi_1^q\xi_3^q-2\xi_1^q\xi_3^q+\xi_2^{q+1}\xi_3^{q-1}-\xi_1\xi_2^{2q}\xi_3^{q-2}
 =\xi_2^{q+1}\xi_3^{q-1}-\xi_1\xi_2^{2q}\xi_3^{q-2}-\xi_1^q\xi_3^q.$$
 Again using Lemmas~\ref{u2_div_lem}, \ref{u2_st}, \ref{u2dlem} and \ref{u2d_st}, $\PP^1(u_2d_{2,2})=u_2\PP^1(d_{2,2})$, $\PP^1(d_{2,2})\in R_3$ and
 $u_2\PP^1(d_{2,2})\equiv_{\langle \xi_1^{q^2}\rangle}\xi_2^q\xi_3^q-\xi_2^q\xi_3^q-\xi_1^q\xi_2^{q^2}=-\xi_1^q\xi_2^{q^2}$. 
 Since $u_2\equiv_{\langle \xi_1^{q+1}\rangle} -\xi_2^{q+1}+\xi_1^q\xi_3$, we have
 $$\PP^1(d_{1,2})\equiv_{\langle \xi_1^{q+1}\rangle}  - \xi_3^{q-1}+\xi_1\xi_2^{q-1}\xi_3^{q-2}$$ and $\PP^1(d_{2,2})\equiv_{\langle \xi_1^{q+1}\rangle} \xi_1^q\xi_2^{q^2-q-1}$.
 Since $\PP^1(d_{1,2})$ and $\PP^1(d_{2,2})$ lie in $R_3$, we have
 $\xi_2^q d_{1,2}-\xi_1^q d_{2,2}\in R_3$.
 Using the congruences we have
  \begin{eqnarray*}
 \xi_3^q&\equiv_{\langle \xi_1^{q+1}\rangle}& 2\xi_2^q d_{1,2}-2\xi_1^q d_{2,2}+
 (2\xi_1\xi_2^{q-1}\xi_3^{q-1}-\xi_1^2\xi_2^{2q-2}\xi_3^{q-2}+\xi_1^q\xi_2^{q^2-q})\cr
  &&-(\xi_3+c_{3,2})(\xi_3^{q-1}-\xi_1\xi_2^{q-1}\xi_3^{q-2})-\xi_1^q\xi_2^{q^2-q}\cr
 &\equiv_{\langle \xi_1^{q+1}\rangle}&2\xi_2^q d_{1,2}-2\xi_1^q d_{2,2}
 +3\xi_1\xi_2^{q-1}\xi_3^{q-1}-\xi_3^q
 -\xi_1^2\xi_2^{2q-2}\xi_3^{q-2}\cr
 &&-c_{3,2}(\xi_3^{q-1}-\xi_1\xi_2^{q-1}\xi_3^{q-2}) \cr
 &\equiv_{\langle \xi_1^{q+1}\rangle}&2\xi_2^q d_{1,2}-2\xi_1^q d_{2,2}-\xi_3^q+3\xi_1\xi_2^{q-1}\xi_3^{q-1}-\xi_1^2\xi_2^{2q-2}\xi_3^{q-2} \cr
 &&-\xi_1\xi_2^{q-1}(\xi_3^{q-1}-\xi_1\xi_2^{q-1}\xi_3^{q-2})\cr
 &\equiv_{\langle \xi_1^{q+1}\rangle}&2\xi_2^q d_{1,2}-2\xi_1^q d_{2,2}-\xi_3^q +2\xi_1\xi_2^{q-1}\xi_3^{q-1}.
 \end{eqnarray*}
 Therefore
 $$\xi_3^q\equiv_{\langle \xi_1^{q+1}\rangle} \xi_2^q d_{1,2}-\xi_1^q d_{2,2}+\xi_1\xi_2^{q-1}\xi_3^{q-1}.$$
 Recall that $- \xi_2^q d_{1,2}+\xi_1^q d_{2,2}\in R_3$ and define
 $r_{2,1}\in R_3$ by 
 \begin{equation} \label{dagger_12}
 \xi_1^{q+1}r_{2,1}=\xi_3^q- \xi_2^q d_{1,2}+\xi_1^q d_{2,2}-\xi_1\xi_2^{q-1}\xi_3^{q-1}.
 \end{equation}
Then 
\begin{equation} \label{dag}
 \xi_3^q=\xi_2^q d_{1,2}-\xi_1^q d_{2,2}-\gamma_{3,2}^{(1)}
 \end{equation}
with
$\gamma_{3,2}^{(1)}=-\xi_1^{q+1}r_{2,1}-\xi_1\xi_2^{q-1}\xi_3^{q-1}\in R_3$
and $\nu(\gamma_{3,2}^{(1)})>q+1$. 

We have completed the proof of Lemma~\ref{oplus_tech-lem} for 
the base case $m=2$.
Parts (a), (c) and (d) follow from Lemmas \ref{u2lem} and \ref{u2dlem}.
Part (b) follows from Lemmas \ref{u2_st}, \ref{u2d_st} and
\ref{special_u2d_st}. Part (e) follows from Lemma~\ref{dg02lem}
and part (f) is Equation~\ref{dag}.
Theorem \ref{oplus_thm_v2} for $m=2$, 
gives the alternative computation for the invariants of $\orthp{4}{q}$.

\begin{remark}
  Notice that in the process of deriving Equation~(\ref{dag}),
  the expression $-\xi_3^q$ arises on the right hand side.  If the coefficient 
  had been +1 rather than -1 then this process would not have provided a
  formula for $\xi_3^q$.   It is for this reason that naively applying a 
  Steenrod operation to the formula for $\xi_{n-i}^{q^i}$  
  is not guaranteed to give a formula for $\xi_{n-1-i}^{q^{i+1}}$.
\end{remark}

\section{Proof of Lemma \ref{oplus_tech-lem}}\label{proof_of_tech-lem}

The proof is by induction on $m$. 
The base case, $m=2$, is given in Section~\ref{base_case}. 
For the induction step, we will assume the lemma is true for $S_\ell^{G_\ell}$ for $2\leq\ell<m$ and, in the proof of a given part, we will assume the
earlier parts are true for $S_m^{G_m}$.

\begin{lem} Part (e) of Lemma~\ref{oplus_tech-lem} follows 
from parts (a) and (c).
\end{lem}
\begin{proof} From (c) we have
$$u_m=(\xi_{n-1}+c_{n-1,m})u_{m-1}^q+\sum_{i=1}^{m-1}(-1)^i(\xi_{n-1-i}+c_{n-1-i,m})(u_{m-1}d_{i,m-1})^q$$
with $c_{j,m}\in R_{j-1}$ and $u_m d_{i,m}-(u_{m-1}d_{i-1,m-1})^q\xi_{n}\in R_{n-1}$.
Therefore
$$F:=\sum_{i=1}^{m}(-1)^{i+1}(\xi_{n-i}+c_{n-i,m})(u_md_{i,m}-(u_{m-1}d_{i-1,m-1})^q\xi_{n})\in R_{n-1}.$$
The coefficient of $\xi_{n}$ in the expression defining $F$ is
$$\sum_{i=1}^{m}(-1)^{i}(\xi_{n-i}+c_{n-i,m})(u_{m-1}d_{i-1,m-1})^q
=-u_m.$$ 
Thus  $$F=-u_m\xi_n+\left(\sum_{i=1}^{m}(-1)^{i+1}(\xi_{n-i}+c_{n-i,m})u_{m}d_{i,m}\right)$$
and $$\xi_n=-F/u_m+\left(\sum_{i=1}^{m}(-1)^{i+1}(\xi_{n-i}+c_{n-i,m})d_{i,m}\right).
$$
Since $F$ is divisible by $u_m$ in $S_m^{G_m}$ and $F\in R_{n-1}$,
we have $F\in \ker(\Phi_{n-1,m-1})$ (see Lemma~\ref{phi_bar_lemma} and 
Remark~\ref{um_div_rem}).
From part (a), $\ker(\Phi_{n-1,m-1})=u_mR_{n-1}$.
Therefore $F$ is divisible by $u_m$ in $R_{n-1}$.
Define $c_{n,m}:=F/u_m\in R_{n-1}$.
This gives
$$\xi_n=-c_{n,m}+\left(\sum_{i=1}^{m}(-1)^{i+1}(\xi_{n-i}+c_{n-i,m})d_{i,m}\right)$$
as required.
\end {proof}

\subsection{Proof of Parts (a), (b) and (c)}
The proof is by induction on $m$.
The base cases $m=2$ is given in Section \ref{base_case}.
By the induction hypothesis $S_{m-1}^{G_{m-1}}$ is generated by $\{\xi_1,\ldots,\xi_{n-3},d_{1,m-1},\ldots, d_{m-1,m-1}\}$.
The Hook subgroup for $G_m$ is a normal subgroup of the pointwise stabiliser $(G_m)_{x_m}$.
If we identify $G_{m-1}$ with the pointwise stabiliser of $\{x_m,y_m\}$ in $G_m$, then $(G_m)_{x_m}=G_{m-1}H$.
Arguing as in the proof of Theorem \ref{sylow_thm} and using Theorem \ref{orHcal}, we see that $S_m^{(G_m)_{x_m}}$ is generated by
$$\{N(y_m),x_m,\xi_1,\ldots,\xi_{n-2},\psi_{[m,1]}(\sigma(d_{1,m-1})),\ldots \psi_{[m,1]}(\sigma(d_{m-1,m-1}))\}.$$
In this context, the compliance relation (Lemma \ref{ortat:3}) allows us to rewrite $\xi_{n-1}$ in terms of the generators of $S_m^{(G_m)_{x_m}}$.
For degree reasons, the resulting expression is linear in the $\psi_{[m,1]}(\sigma(d_{j,m-1}))$, giving
$$\xi_{n-1}=x_mN(y_m)+\lambda_0+\sum_{j=1}^{m-1}\lambda_j \psi_{[m,1]}(\sigma(d_{j,m-1}))$$
with $\lambda_j\in R_{n-2}[x_m]$.
Since $u_m=x_mN(y_m)\psi_{[m,1]}(\sigma(u_{m-1}))$, we have
$$u_m=\psi_{[m,1]}(\sigma(u_{m-1}))\left(\xi_{n-1}-\lambda_0-\sum_{j=1}^{m-1}\lambda_j \psi_{[m,1]}(\sigma(d_{j,m-1}))\right).$$
We know that $\psi_{[m,1]}$ and $\sigma$ are algebra maps. Furthermore, $\psi_{[m,1]}(\sigma(f))=\psi_{[m,1]}(f)$ for $f\in R_{n-2}$.
By induction $u_{m-1}$ and $u_{m-1}d_{j,m-1}$lie in $R_{n-2}$. Therefore
$$u_m=(\xi_{n-1}-\lambda_0)\psi_{[m,1]}(u_{m-1})-\sum_{j=1}^{m-1}\lambda_j \psi_{[m,1]}(u_{m-1}d_{j,m-1}).$$
Write $\lambda_j=\overline{\lambda_j}+x_m\lambda_j'$ with $\overline{\lambda_j}\in R_{n-2}$ and $\lambda_j'\in R_{n-2}[x_m]$
and observe that $\psi_{[m,1]}(\lambda_j)=\psi_{[m,1]}(\overline{\lambda_j})$.
Using Lemma \ref{phi_bar_lemma}, $\psi_{[m,1]}(\psi_{[m,1]}(f))=\psi_{[m,1]}(f)^q$ for $f\in R_{n-2}$.
Applying $\psi_{[m,1]}$ to the above expression and using the fact that $\psi_{[m,1]}(u_m)=0$ gives
$$0=\psi_{[m,1]}(\xi_{n-1}-\overline{\lambda_0})\psi_{[m,1]}(u_{m-1})^q-\sum_{j=1}^{m-1}\psi_{[m,1]}(\overline{\lambda_j}) \psi_{[m,1]}(u_{m-1}d_{j,m-1})^q.$$
Therefore
\begin{equation}\label{ut_eqn}
    \widetilde{u}_m:=(\xi_{n-1}-\overline{\lambda_0})u_{m-1}^q-\sum_{j=1}^{m-1}\overline{\lambda_j} (u_{m-1}d_{j,m-1})^q\in R_{n-1}\cap\ker(\psi_{[m,1]}).
\end{equation}
Using Lemma \ref{phi_bar_lemma}, $u_m$ divides every element 
of $R_{n-1}\cap\ker(\psi_{[m,1]})$. Therefore $u_m$ divides $\widetilde{u}_m$.
Since $\wdeg(u_m)=\wdeg(\widetilde{u}_m)$, we see that  $\widetilde{u}_m$
is a scalar multiple of $u_m$.
As an element of $R_{n-1}$, $\widetilde{u}_m$ is linear in $\xi_{n-1}$
with coefficient $u_{m-1}^q\not=0$. Therefore $\widetilde{u}_m$
is a non-zero scalar multiple of $u_m$.
From this we conclude that $\ker(\Phi_{n-1,m-1})=u_mR_{n-1}$.
We claim that that $\widetilde{u}_m=u_m$. To see this, first observe that
$\lt(u_m)=\prod_{i=1}^m y_i^{q^{m+i-2}} x_i^{q^{m-i}}= : \beta$ using either the lex or grevlex order on $S_m$. We will show that $\beta$ appears in $\widetilde{u}_m$ with coefficient one. The factor $y_m^{q^{n-2}}x_m$
can only come from a factor of $\xi_{n-1}$. Therefore, using that the expression for $\widetilde{u}_m$ as an element of $R_{n-1}$, we see that $\beta$ can only appear as a term in $\xi_{n-1}u_{m-1}^q$. The factor $y_{m-1}^{n-3}x_{m-1}^q$ necessarily comes 
from a factor of $\xi_{n-3}^q$ appearing in a term in $u_{m-1}^q$.
Iterating this process and using induction, we see that
$\xi_{n-1}\xi_{n-3}^q\cdots\xi_{1}^{q^{m-1}}$ appears with coefficient one in $\widetilde{u}_m\in R_{n-1}$ and $\beta$ appears with coefficient one in $\widetilde{u}_m\in S_m$. Therefore $\widetilde{u}_m=u_m$.
Applying $\PP^{e(i,m)}$ to this expression shows that
$u_md_{i,m}=\PP^{e(i,m)}(u_m)\in R_{n}$ and $u_md_{i,m}-\xi_{n}\PP^{e(i,m)-q^{n-2}}(u_{m-1}^q)\in R_{n-1}$.
Since $e(i,m)-q^{n-2}=qe(i-1,m-1)$, we have $u_md_{i,m}-\xi_{n}(u_{m-1}d_{i-1,m-1})^q\in R_{n-1}$.

As an element of $S_m$, $u_m$ is a product of pairwise relatively prime linear forms. Therefore $u_m$ divides $\PP^i(u_m)$ in $S_m$.
If $\PP^i(u_m)\in R_{n-1}$, then 
$$\PP^i(u_m)\in  \ker(\Phi_{n-1,m-1})=u_mR_{n-1}$$ and $u_m$ divides 
$\PP^i(u_m)$ in $R_{n-1}$(see Remark~\ref{um_div_rem}).
For $i<q^{n-2}$, $\PP^i(u_m)\in R_{n-1}$ and, therefore, $\PP^i(u_m)/u_m\in R_{n-1}$ .
By induction, $\PP^i(u_{m-1})=0$ for $0<i<q^{m-2}$. Therefore, for $0<i<q^{m-1}$, $\PP^i(u_{m-1}^q)=0$ and $\PP^i(u_m)\in R_{n-2}$.
Since $u_m$ is linear in $\xi_{n-1}$ and divides $\PP^i(u_m)\in R_{n-2}$, we conclude that $\PP^i(u_m)=0$ for $0<i<q^{m-1}$.

For $i<q^{n-1}$, write $i=\ell q^{n-2}+r$ with $\ell<q$ and $r<q^{n-2}$. 
We will prove by induction on $i$ that $u_m^{\ell+1}\PP^i(d_{j,m})\in R_{n}$ with $\xi_{n}$-degree at most $\ell+1$.
For $i=0$, we have $\ell=0$ and $u_m\PP^0(d_{j,m})=u_md_{j,m}\in R_{n}$ with $\xi_{n}$-degree $1$.
For $0<i<q^{m-1}$, we have $u_m\PP^i(d_{j,m})=\PP^i(u_m d_{j,m})\in R_{n-1}$, giving $\PP^i(d_{j,m})\in R_{n-1}$ (see Remark~\ref{um_div_rem}).
For $i\geq q^{m-1}$, using the Cartan identity gives
$$u_m\PP^i(d_{j,m})=\PP^i(u_m d_{j,m})-\sum_{k>0}\PP^k(u_m)\PP^{i-k}(d_{j,m}).$$
Since $u_m d_{j,m}\in R_{n}$ with $\xi_{n}$-degree $1$, we see that $\PP^i(u_m d_{j,m})\in R_{n}$ with $\xi_{n}$-degree at most $\ell+1$.
For $k>0$, write $i-k=\ell'q^{n-2}+r'$ with $\ell'\leq \ell$ and $r'<q^{n-2}$. Since $i-k<i$, the induction hypothesis gives
$u_m^{\ell'+1}\PP^{i-k}(d_{j,m})\in R_{n}$ with $\xi_{n}$-degree at most $\ell'+1$.
If $k<q^{n-2}$, then $\PP^k(u_m)/u_m\in R_{n-1}$ and 
$$u_m^{\ell}\PP^k(u_m)\PP^{i-k}(d_{j,m}) =(\PP^k(u_m)/u_m)u_m^{\ell+1}\PP^{i-k}(d_{j,m})\in R_{n}$$
with $\xi_{n}$-degree at most $\ell+1$, as required. Suppose $q^{n-2}\leq k$. Then $\ell'<\ell$ and 
$u_m^{\ell}\PP^{i-k}(d_{j,m})\in R_{n}$ with $\xi_{n}$-degree at most $\ell$.
Since $\PP^k(u_m)\in R_{n}$ with $\xi_{n}$-degree at most $1$,
$u_m^{\ell}\PP^k(u_m)\PP^{i-k}(d_{j,m}) \in R_{n}$ with $\xi_{n}$-degree at most $\ell+1$, as required.

Suppose $i=(q-1)q^{n-2}+r$ with $r<q^{n-2}$ and $i\not\equiv_{(q)}0$. 
We know that $u_m^q\PP^i(d_{j,m})\in R_{n}$ with $\xi_{n}$-degree at most $q$.
We will show that the $\xi_{n}$-degree at most $q-1$.
Using the Cartan identity
$$u_m^q\PP^i(d_{j,m})=\PP^i(u_m^qd_{j,m})-\sum_{k>0}\PP^k(u_m)^q\PP^{i-qk}(d_{j,m}).$$
Since $u_m^qd_{j,m}\in R_{n}$ with $\xi_{n}$-degree $1$, we see that $\PP^i(u_m^qd_{j,m})\in R_{n}$
with $\xi_{n}$-degree at most $q$. However, the coefficient of $\xi_{n}$ in $u_m^qd_{j,m}$
is given by $u_m^{q-1}(u_{m-1}d_{j-1,m-1})^q$, which means that the coefficient of $\xi_{n}\xi_{n-1}^{q-1}$ in  
 $u_m^qd_{j,m}$ is $(u_{m-1}^q d_{j-1,m-1})^q$. 
 Therefore we have that the coefficient of $\xi_{n}^q$ in $\PP^i(u_m^qd_{j,m})$
 is given by $\PP^r((u_{m-1}^q d_{j-1,m-1})^q)$. Since $r\not\equiv_{(q)} 0$, $\PP^r((u_{m-1}^q d_{j-1,m-1})^q)=0$
 and the $\xi_{n}$-degree of $\PP^i(u_m^qd_{j,m})$ is at most $q-1$. Since $i<q^{n-1}$ and $qk<i$, we have $k<q^{n-2}$.
 Thus $\PP^k(u_m)/u_m\in R_{n-1}$ and $$\PP^k(u_m)^q\PP^{i-qk}(d_{j,m})=(\PP^k(u_m)/u_m)^q u_m^q \PP^{i-qk}(d_{j,m}).$$
 Write $r=r_0+qr'$ with $0<r_0<q$. If $r'=0$, then $i-kq<(q-1)q^{n-2}$ and $u_m^q \PP^{i-qk}(d_{j,m})$ has $\xi_{n}$-degree at most $q-1$.
 If $r'>0$, then the result follows by induction on $r'$.

Using the induction hypothesis, in $S_{m-1}^{G_{m-1}}$, we have
 $$\xi_{n-2}=\left(\sum_{i=1}^{m-1}(-1)^{i+1}(\xi_{n-2-i}+c_{n-2-i,m-1})d_{i,m-1}\right)-c_{n-2,m-1}.$$
 Using Lemma \ref{sub_max}, 
 $$\PP^{q^{n-3}}(\xi_{n-2-i}d_{i,m-1})=\xi_{n-1-i}d_{i,m-1}^q+\xi_{n-2-i}^q\PP^{a_i}(d_{i,m-1})$$
and
 $$\PP^{q^{n-3}}(c_{n-2-i,m-1}d_{i,m-1})=\PP^{q^{n-3-i}}(c_{n-2-i,m-1})d_{i,m-1}^q+c_{n-2-i,m-1}^q\PP^{a_i}(d_{i,m-1})$$
 where $a_i=\deg(d_{i,m-1})-1$. Therefore, applying $\PP^{q^{n-3}}$ to the relation above gives
 \begin{eqnarray*}
 \xi_{n-1}&=&\left(\sum_{i=1}^{m-1}(-1)^{i+1}(\xi_{n-1-i}+\PP^{q^{n-3-i}}(c_{n-2-i,m-1}))d_{i,m-1}^q\right)\\
 &&+ \left(\sum_{i=1}^{m-1}(-1)^{i+1}(\xi_{n-2-i}+c_{n-2-i,m-1})^q \PP^{a_i}(d_{i,m-1})\right)-\PP^{q^{n-3}}(c_{n-2,m-1}).
 \end{eqnarray*}
Multiplying by $u_{m-1}^q$ gives
\begin{eqnarray*}
 u_{m-1}^q\xi_{n-1}&=&\left(\sum_{i=1}^{m-1}(-1)^{i+1}(\xi_{n-1-i}+\PP^{q^{n-3-i}}(c_{n-2-i,m-1}))(u_{m-1}d_{i,m-1})^q\right)\\
 &&+ \left(\sum_{i=1}^{m-1}(-1)^{i+1}(\xi_{n-2-i}+c_{n-2-i,m-1})^q u_{m-1}^q\PP^{a_i}(d_{i,m-1})\right)\\
 &&-u_{m-1}^q\PP^{q^{n-3}}(c_{n-2,m-1}).
 \end{eqnarray*}
By induction $u_{m-1}d_{i,m-1}\in R_{n-2}$ with $\xi_{n-2}$-degree $1$
and, since $a_i=q^{n-3}-q^{n-3-i}-1$, $u_{m-1}^q\PP^{a_i}(d_{i,m-1})\in R_{n-2}$ with $\xi_{n-2}$-degree at most $q-1$.
Therefore the right hand side of the relation is an element of $R_{n-2}$ and the relation lifts to give an element of
$\ker(\Phi_{n-1,m-1})=u_mR_{n-1}\subset S_m^{G_m}$. 
Comparing the coefficient of $\xi_{n-1}$ gives
\begin{eqnarray*}
 u_m&=&u_{m-1}^q\xi_{n-1}-\left(\sum_{i=1}^{m-1}(-1)^{i+1}(\xi_{n-1-i}+\PP^{q^{n-3-i}}(c_{n-2-i,m-1}))(u_{m-1}d_{i,m-1})^q\right)\\
 &&-\left(\sum_{i=1}^{m-1}(-1)^{i+1}(\xi_{n-2-i}+c_{n-2-i,m-1})^q u_{m-1}^q\PP^{a_i}(d_{i,m-1})\right)\\
 &&+u_{m-1}^q\PP^{q^{n-3}}(c_{n-2,m-1}).
 \end{eqnarray*}
Comparing this expression for $u_m$ with the expression coming from Equation~\ref{ut_eqn}
and using the fact that $u_{m-1}^q\PP^{a_i}(d_{i,m-1})$ 
has $\xi_{n-2}$-degree at most $q-1$ while 
$(u_{m-1}d_{i,m-1})^q$ has $\xi_{n-2}$-degree $q$, we conclude that
$$u_m=(\xi_{n-1}+c_{n-1,m})u_{m-1}^q+\sum_{i=1}^{m-1}(-1)^i(\xi_{n-1-i}+c_{n-1-i,m})(u_{m-1}d_{i,m-1})^q$$
with $$ c_{n-1,m}=\PP^{q^{n-3}}(c_{n-2,m-1})+\sum_{i=1}^{m-1}(-1)^{i}(\xi_{n-2-i}+c_{n-2-i,m-1})^q \PP^{a_i}(d_{i,m-1})\in R_{n-2}$$
and $c_{n-1-i,m}=\PP^{q^{n-3-i}}(c_{n-2-i,m-1})\in R_{n-2-i}$ for $0<i<m$. This completes the proof of parts (a), (b) and (c).   \qed

\subsection{Proof of Part (d)}
The proof is by induction on $m$. For $m=2$, the result follows from  Lemmas \ref{u2lem} and \ref{u2dlem}.

Suppose $m>2$.
By induction, $u_{m-1}=M(0,m-1)+\delta_{0,m-1}$ and $u_{m-1}d_{i,m-1}=M(i,m-1)+\delta_{i,m-1}$.
Substituting into the expression for $u_m$ from part (c) gives
\begin{eqnarray*}
u_m&=&(\xi_{n-1}+c_{n-1,m})(M(0,m-1)+\delta_{0,m-1})^q\\
     &&+\sum_{i=1}^{m-1}(-1)^i(\xi_{n-1-i}+c_{n-1-i,m})(M(i,m-1)+\delta_{i,m-1})^q\\
     &=&\sum_{i=0}^{m-1}(-1)^i\xi_{n-1-i}M(i,m-1)^q+\sum_{i=0}^{m-1}(-1)^i\xi_{n-1-i}\delta_{i,m-1}^q\\
     &&+\sum_{i=0}^{m-1}(-1)^ic_{n-1-i,m}(M(i,m-1)^q+\delta_{i,m-1}^q).
\end{eqnarray*}     
Observe that $M(0,m)=\sum_{i=0}^{m-1}(-1)^i\xi_{n-1-i}M(i,m-1)^q$ and define
$$\delta_{0,m}=\sum_{i=0}^{m-1}(-1)^i\xi_{n-1-i}\delta_{i,m-1}^q
     +\sum_{i=0}^{m-1}(-1)^ic_{n-1-i,m}(M(i,m-1)^q+\delta_{i,m-1}^q).$$
Then $u_m=M(0,m)+\delta_{0,m}$ with $\delta_{0,m}\in R_{n-1}$.  
Since $c_{j,m}\in R_{j-1}$ and $\wdeg(c_{j,m})=\wdeg(\xi_{j})=q^{j-1}+1$, 
we see that $\nu(c_{j,m})>1$.
Hence $\nu(M(i,m-1))=1+q+\cdots+q^{m-2}$ and $\nu(\delta_{i,m-1})>1+q+\cdots+q^{m-2}$,
we have $\nu(\delta_{0,m})>1+q+\cdots+q^{m-1}$.

Apply $\PP^{e(i,m)}$ to $u_m$ to get 
$u_md_{i,m}=\PP^{e(i,m)}(M(0,m))+\PP^{e(i,m)}(\delta_{0,m})$.
Define $\delta_{i,m}:=u_m d_{i,m}-M(i,m)$ so that $u_m d_{i,m}=M(i,m)+\delta_{i,m}$, and observe that $\delta_{i,m}\in R_{n}$.
Using Lemma \ref{nu_st_lem}, $$\nu(\PP^{e(i,m)}(\delta_{0,m}))\geq \nu(\delta_{0,m})>1+q+\cdots+q^{m-1}.$$
Using Lemma \ref{nu_min_lem}, $\nu(\PP^{e(i,m)}(M(0,m)-M(i,m))>1+q+\cdots+q^{m-1}$.
Therefore $\nu(\delta_{i,m})>1+q+\cdots +q^{m-1}$.

It is clear that using the weighted revlex order on $R_{n}$, 
$$\lt(M(0,m))=(-1)^{\lfloor m/2 \rfloor}\xi_{m}^{q^{m-1}+q^{m-2}+\cdots+1}$$ and 
$$\lt(M(i,m))=(-1)^{\lfloor m/2 \rfloor}\xi_{m}^{q^{m-1-i}+\cdots+1}\xi_{m+1}^{q^{m-1}+\cdots+q^{m-i}}.$$
From Lemma \ref{nu_ord_lem}, $\lt(\delta_{i,m})<\lt(M(i,m))$ for $0\leq i\leq m$. 
Therefore $\lt(u_m)=\lt(M(0,m))$ and $\lt(u_m d_{i m})=\lt(M(i,m))$, completing the proof part (d).  \qed

\subsection{Proof of Part (f)}
The proof is by induction on $m$. 
The base case, $m=2$, is given in Section~\ref{base_case}. 
Suppose $m>2$. By induction we have parts (e) and (f) of the lemma for 
$S_{m-1}^{G_{m-1}}$. 
Therefore, for $1\leq i \leq m-1$, we have the relation
$$ 0=u_{m-1}\left(\xi_{n-1-i}^{q^{i-1}}+\gamma_{n-1-i,m-1}^{(i-1)}-\sum_{j=1}^{m-1}(-1)^{j+1}(\xi_{n-1-i-j}^{q^{i-1}}+\gamma_{n-1-i-j,m-1}^{(i-1)})d_{j,m-1}\right)$$
in $R_{n-2}$. Define
$$\gamma^{(i)}:=\sum_{j=0}^{m-1}(-1)^{j}(\xi_{n-1-i-j}^{q^{i-1}}+\gamma_{n-1-i-j,m-1}^{(i-1)})^q u_md_{j+1,m}.$$
By part (c), $u_m d_{j,m}-(u_{m-1}d_{j-1,m-1})^q\xi_{n}\in R_{n-1}$ for $1\leq j\leq m$, using the convention $d_{0,m-1}=1$.
Therefore $\gamma^{(i)}\in R_{n}$. Furthermore $\gamma^{(i)}$ is linear in $\xi_{n}$ with coefficient
$$\left(-\sum_{j=0}^{m-1}(-1)^{j}(\xi_{n-1-i-j}^{q^{i-1}}+\gamma_{n-1-i-j,m-1}^{(i-1)})u_{m-1}d_{j,m-1}\right)^q=0.$$
Hence $\gamma^{(i)}\in R_{n-1}$. From part (a), $\ker(\Phi_{n-1,m-1})=u_mR_{n-1}$. 
Since $u_m$ divides $\gamma^{(i)}$ in $S_m^{G_m}$ and $\gamma^{(i)}$ lies in $R_{n-1}$, we see that $\gamma^{(i)}\in u_mR_{n-1}$ (see Remark~\ref{um_div_rem}). Thus $\gamma^{(i)}/u_m\in R_{n-1}$.

Define an $(m+1)\times(m+1)$ matrix over $R_{n}$ by
$$M_m^{(i)}:=\begin{pmatrix}
\xi_{n-i}^{q^i}&\cdots &\xi_{m-i}^{q^i}\\
&M_m&
\end{pmatrix}.$$
Since row $1$ and row $i+1$ of $M_m^{(i)}$ are equal, $\det(M_m^{(i)})=0$. Therefore, computing the determinant by expanding across the first row gives
$$\sum_{j=0}^m (-1)^j\xi_{n-i-j}^{q^i} M(j,m)=0.$$
Hence
$$\xi_{n-i}^{q^i}M(0,m)=\sum_{j=1}^m (-1)^{j+1}\xi_{n-i-j}^{q^i} M(j,m).$$
From part (d), $u_m d_{j,m}=M(j,m)+\delta_{j,m}$ with $\nu(\delta_{j,m})>1+q+\cdots +q^{m-1}$.
Substituting into the expression for $\gamma^{(i)}$ gives 
\begin{eqnarray*}
\gamma^{(i)}&=&\sum_{j=0}^{m-1}(-1)^{j}\xi_{n-1-i-j}^{q^{i}}M(j+1,m)
+\sum_{j=0}^{m-1}(-1)^{j}(\gamma_{n-1-i-j,m-1}^{(i-1)})^q M(j+1,m)\\
&&+\sum_{j=0}^{m-1}(-1)^{j}(\xi_{n-1-i-j}^{q^{i-1}}+\gamma_{n-1-i-j,m-1}^{(i-1)})^q \delta_{j+1,m}\\
&=&\xi_{n-i}^{q^i}M(0,m)+\sum_{j=0}^{m-1}(-1)^{j}(\gamma_{n-1-i-j,m-1}^{(i-1)})^q M(j+1,m)\\
&&+\sum_{j=0}^{m-1}(-1)^{j}(\xi_{n-1-i-j}^{q^{i-1}}+\gamma_{n-1-i-j,m-1}^{(i-1)})^q \delta_{j+1,m}.
 \end{eqnarray*}
 Therefore
\begin{eqnarray*}
\gamma^{(i)}-\xi_{n-i}^{q^i}u_m&=&-\xi_{n-i}^{q^i}\delta_{0,m}+\sum_{j=0}^{m-1}(-1)^{j}(\gamma_{n-1-i-j,m-1}^{(i-1)})^q M(j+1,m)\\
&&+\sum_{j=0}^{m-1}(-1)^{j}(\xi_{n-1-i-j}^{q^{i-1}}+\gamma_{n-1-i-j,m-1}^{(i-1)})^q \delta_{j+1,m}.
\end{eqnarray*}
By induction, $\nu(\gamma_{j+1,m-1}^{(i-1)})>q^{i-1}$. Therefore $$\nu(\gamma^{(i)}-\xi_{n-i}^{q^i}u_m)> q^i+1+q+\cdots +q^{m-1},$$
giving $\nu(\gamma^{(i)}/u_m-\xi_{n-i}^{q^i})>q^i$. Define $\gamma_{n-i,m}^{(i)}:=\gamma^{(i)}/u_m-\xi_{n-i}^{q^i}$ and, for $m-1-i\leq\ell<n-1-i$,
define $\gamma_{\ell+1,m}^{(i)}:=(\gamma_{\ell+1,m-1}^{(i-1)})^q$. Then
$$\gamma^{(i)}/u_m=\sum_{j=0}^{m-1}(-1)^{j}(\xi_{n-1-i-j}^{q^{i}}+\gamma_{n-1-i-j,m}^{(i)}) d_{j+1,m}$$
and
\begin{eqnarray*}
\xi_{n-i}^{q^i}&=&\gamma^{(i)}/u_m-\gamma_{n-i,m}^{(i)}\\
&=&\left(\sum_{j=0}^{m-1}(-1)^{j}(\xi_{n-1-i-j}^{q^{i}}+\gamma_{n-1-i-j,m}^{(i)}) d_{j+1,m}\right)-\gamma_{n-i,m}^{(i)}\\
&=&\left(\sum_{j=1}^{m}(-1)^{j+1}(\xi_{n-i-j}^{q^{i}}+\gamma_{n-i-j,m}^{(i)}) d_{j,m}\right)-\gamma_{n-i,m}^{(i)}
\end{eqnarray*}
with $\nu(\gamma_{k,m}^{(i)})>q^i$, as required.  \qed

\section{Generation Over The Steenrod Algebra}

Let $\A$ denote the algebra of operations on $S_m$ generated by the Steenrod operations with the product given by composition.
\begin{thm} \label{st_alg_gen}
    $S_m^{G_m}$ is generated as an $\A$-algebra by 
    the two elements $\xi_1$ and $d_{1,m}$.
\end{thm}
\begin{proof}
    Let $C$ denote the smallest algebra, closed under the action of the
    Steenrod algebra and containing $\{\xi_1,d_{1,m}\}$.
    Clearly $C \subseteq S_m^{G_m}$.  
    Since $\PP^{q^{i-1}}(\xi_{i})=\xi_{i+1}$ we see that $\xi_{i+1} \in C$ for all $i \geq 1$.
    It remains to show that $d_{i,m} \in C$ for $i=1,2,3,\dots,m$.
    We proceed by induction.  
     Suppose then that
    $d_{1,m},d_{2,m},\dots,d_{i-1,m} \in C$ and consider $F:=\PP^{q^{n-i-1}}(d_{i-1,m})$ of degree $q^{n-1}-q^{n-1-i}$.
    Recall, from the proof of Theorem~\ref{ortho_hsop_thm} that 
    $d_{i-1,m}  \equiv_I \pm d_{i-1}(W_m)^{q^{m-1}}$
    where $I$ is the ideal 
    $I=\langle x_1,x_2,\ldots,x_m\rangle$, $W_m={\rm Span}_{\field_q}\{y_m,\ldots,y_1\}$,
    and $d_{i-1}(W_m)$ is the $(i-1)^\text{th}$ Dickson invariant in the variables
    $y_m,\ldots,y_1$. 
    Since $\PP^{q^{m-i}}(d_{i-1}(W_m)) = -d_i(W_m)$ (see 
    \cite{Wilkerson-primDickinva:83}) it follows that 
    $F \equiv_I \PP^{q^{n-i-1}}(d_{i-1}(W_1)^{q^{m-1}}) = -d_i(W_1)^{q^{m-1}} \neq 0$.
    But the only invariants of degree $q^{n-1}-q^{n-1-i}$ which do not vanish on $W_m$
    are scalar multiples of $d_{i,m}$.  Hence $F = \pm d_{i,m} + f$ where 
    $f \equiv_I 0$ and 
    $f \in \field_q[\xi_1,\ldots,\xi_{m},d_{1,m},\ldots,d_{i-1,m}] \subseteq C$.  
    This shows that $d_{i,m} \in C$ as required.
\end{proof}

Here is an alternate proof of the theorem.

\begin{proof}
  As in the proof of Theorem \ref{ortho_hsop_thm}, let $I$ denote the ideal in $S_m$ generated by $\{x_1, \ldots, x_m\}$. 
  Let $J$ denote the ideal in $S_m^{G_m}$ generated by 
  $\{\xi_1, \ldots, \xi_{n-1}\}$. Since $J \subset I$, the inclusion of
  $S_m^{G_m}$ into $S_m$ induces a map from $S_m^{G_m}/J$ to $S_m/I$. Since both $I$ and $J$ are closed under the action of the Steenrod algebra, this is a map of $\A$-algebras. Let $D$ denote the subalgebra of $S_m$ generated by $\{d_{1,m},\ldots, d_{m,m}\}$, let 
  $\widetilde{D}$ denote the image of $D$ in $S_m^{G_m}/J$ and let $\overline{D}$ denote the image of $D$ in $S_m/I$. 
  It follows from Theorem~\ref{ortho_hsop_thm} and its proof that 
  $D\cap J\subset D\cap I=0$. Therefore,
  as algebras $D$, $\widetilde{D}$, and $\overline{D}$ are isomorphic. Hence the map from $S_m^{G_m}/J$ to $S_m/I$ restricts to an isomorphism from $\widetilde{D}$ to $\overline{D}$ as $\A$-algebras. 
  From the proof of Theorem~\ref{ortho_hsop_thm}, $\overline{D}$ is isomorphic to $(\field_q[W_m^*]^{\gl{m}{q}})^{q^{m-1}}$. Since $d_1(W_m)$ generates $\field_q[W_m^*]^{\gl{m}{q}}$ (see \cite{Wilkerson-primDickinva:83}), $\overline{d_1}$ generates $\overline{D}$ and $\tilde{d_1}$ generates  $\widetilde{D}$, as $\A$-algebras. 
  From this we see that $\{\xi_1, d_{1,m}\}$ generates $S_m^{G_m}$ as an $\A$-algebra.  
\end{proof}

\begin{cor} $\invring$ is generated as an $\A$-algebra by 
$\{\xi_1,\R(N(y_1)^{q-1})\}$.
\end{cor}
\begin{proof}
Since $\deg(\R(N(y_1)^{q-1}))=\deg(d_{1,m})<\deg(d_{i,m})$ for $i>1$,
it follows from Theorem \ref{oplus_thm_v2}, that 
$\R(N(y_1)^{q-1})=cd_{1,m}+\delta$ for some scalar $c$ and $\delta\in R_{n-1}$.
By Theorem~\ref{ren_thm} and Lemma~\ref{lex_lt_lem} 
$$\lt(\R(N(y_1)^{q-1}))=y_1^{(q-1)q^{n-2}}=\lt(d_{1,m})$$
using the lexicographic order on $S_m$. Since $\delta\in I$, we see that $c=1$,
and the result follows from Theorem \ref{st_alg_gen}.
\end{proof}

\section*{Acknowledgement}
  The computer algebra package Magma \cite{magma:97} was very useful in this work.  It was
used to test and confirm many hypotheses in low dimensions for small primes.
We are grateful to Karl Dilcher and Keith Taylor (Canadian Mathematical Society book editors) and Donna Chernyk (Springer) for encouraging us to work on this problem.
We thank the referee for a careful reading and constructive criticism.

\end{document}